\documentclass[preprint]{elsarticle}
\usepackage{lineno,hyperref}
\modulolinenumbers[5]

\usepackage{amsmath}
\usepackage{amssymb}
\usepackage{amsthm}
\usepackage{graphicx}
\usepackage{epic,eepic,epsfig}
\usepackage{color}
\usepackage{subfigure}  
\usepackage{placeins}   
\usepackage{multirow}
\usepackage{epstopdf}
\usepackage{varwidth}
\usepackage[table]{xcolor}

\newtheorem{theorem}{Theorem}[section]
\newtheorem{example}{Example}[section]
\newtheorem{lemma}{Lemma}[section]
\newtheorem{remark}{Remark}[section]
\newtheorem{corollary}{Corollary}[section]

\definecolor{tabclr}{cmyk}{0,0,1,0}

\begin{document}

\title{Error estimates on a finite volume method for diffusion problems with interface on Eulerian grids}

\author[XTU1]{Jie Peng}
\ead{xtu\_pengjie@163.com}

\author[XTU1,XTU2]{Shi Shu}
\ead{shushi@xtu.edu.cn}

\author[XTU1]{HaiYuan Yu\corref{cor}}
\ead{spring@xtu.edu.cn}

\author[XTU1,XTU3]{Chunsheng Feng}
\ead{fengchsh@xtu.edu.cn}

\author[YS]{Mingxian Kan}
\ead{kanmx@caep.cn}

\author[YS]{Ganghua Wang}
\ead{wanggh@caep.cn}

\cortext[cor]{Corresponding author}
\address[XTU1]{School of Mathematics and Computational Science, Xiangtan University, Xiangtan 411105, China}
\address[XTU2]{Hunan Key Laboratory for Computation and Simulation in Science and Engineering,
Xiangtan University, Xiangtan 411105, China}
\address[XTU3]{Guangdong Provincial Engineering Technology Research Center for Data Science, Guangzhou 510631, China}
\address[YS]{Institute of Fluid Physics, CAEP, P. O. Box 919 105, Mianyang 621900, China}

\begin{abstract}
The finite volume methods are frequently employed in the
discretization of diffusion problems with interface. In this paper,
we firstly present a vertex-centered MACH-like finite volume method
for solving stationary diffusion problems with strong discontinuity
and multiple material cells on the Eulerian grids. This method is
motivated by Frese  [No. AMRC-R-874, Mission Research Corp.,
Albuquerque, NM, 1987]. Then, the local truncation error and global
error estimates of the degenerate five-point MACH-like scheme are
derived by introducing some new techniques. Especially under some
assumptions, we prove that this scheme  can reach the asymptotic
optimal error estimate $O(h^2 |\ln h|)$   in the maximum norm.
Finally, numerical experiments verify theoretical results.
\end{abstract}


\maketitle

\section{Introduction}\setcounter{equation}{0}

Diffusion problems with interface are
 widely applied in multi-fluid hydrodynamic, fluid-solid coupling mechanics and many other
scientific and engineering computation fields.
The finite volume method (FVM), which presents local conservation
and obvious advantage to handle physical models with complex characteristics very well,
becomes an important discretization method for solving partial differential equations. 

The finite volume methods (FVMs) based on Lagrangian and Eulerian
grids are two commonly used methods for solving diffusion problems.
The moving interface can  be accurately described as we use the
former, but the calculation is hard to execute on highly distorted
grids. There are a lot of researches interested in these FVMs, e.g.
\cite{1981DSK,2000FH,2006JD,2007YGW,2012RAK,2014JD}. An advantage of
the FVMs in the Eulerian frame is the reasonable shape of the
computational grids, such as the uniform grids. But for the
diffusion problems with strong discontinuity on the internal
interface, the difficulty lies in dealing with the cells involving
multiple material properties (\cite{2008SYK}). Many researchers
investigated this kind of FVMs,  e.g.
\cite{1999RE,2015LZ,2006SS,2013NCY,2013ML,2009MCR,2006MO,2001RE,2001LR,2015WT,2010LYH,2012LYH}.
Ewing constructed an immersed finite volume element method (FVEM) on
a uniform triangle grid, and presented the optimal error estimate in
the energy norm (\cite{1999RE,2015LZ}).
Shu  
established the superconvergence theory of a bilinear FVEM  for
diffusion problems with smooth coefficients on the rectangular grids
(\cite{2006SS,2013NCY}).
Ma presented a recovery-based posterior error estimator for FVMs to solve elliptic interface problems in \cite{2013ML}.
There also have a lot of  works on the quadrilateral grids.
For example, Oevermann focused on the two-dimensional (2D) and three-dimensional (3D) problems with
discontinuous fluxes across the interface, and constructed a linear immersed FVEM in \cite{2009MCR} and \cite{2006MO}.
Ewing constructed a series of FVMs with different average methods for diffusion problems under homogeneous jump conditions,
and numerical experiments were carried on to confirm the approximation order of these methods in \cite{2001RE}.
Luce derived a kind of FVM by using a decomposition technique in \cite{2001LR}.
Wang constructed a fourth-order compact FVM and derived some high accuracy post-processing formulas in \cite{2015WT}.
Li proved the optimal $L^2$ error estimate for  bilinear and biquadratic FVMs with smooth coefficients
under a mesh restriction of $h^2$-parallelogram, respectively (\cite{2010LYH,2012LYH}).
In addition, there are many works for the error estimates of finite element method to solve diffusion
problems with smooth coefficients,  for example,  Nie derived the optimal
$L^2$ error estimate of linear finite element scheme for the
diffusion problem with nonlocal boundary in \cite{2014FEM}. However, there have been
few strict theoretical analyses of the global error estimation for
Eulerian FVMs to solve diffusion problems with strong discontinuity
and multiple material cells.


%
%

In the late 1980's, M.H. Frese presented a finite volume method for diffusion equations in 2D magnetohydrodynamic problems
on the quadrilateral grids (\cite{MHF1987}). The corresponding software packages named MACH2 and MACH3 for 2D and 3D problems have been successfully developed, respectively, by
Philips Lab/WSP, Kirtland Air Force Base (\cite{MHF1987,SHMP1993}), and widely used in the numerical simulations of liner implosion system, plasma thrusters and so on (see, e.g. \cite{s-0001,PFS1998,l-0001,m-0001,2013DREW}).
For simplicity of presentation, we denote this finite volume method as MACH FVM.
However, to our knowledge, we have not found the strict error theories of it for diffusion problems with strong discontinuity and multiple material cells on the Eulerian grids,
which urges us to study it.

In this paper, we present a vertex-centered MACH-like FVM on the quadrilateral grids for
stationary diffusion problems with interface.
In particular, for the square grids, this kind of nine-point scheme is degenerated into a five-point scheme.
It is worth pointing out that many classical nine-point FVMs (see, for example, \cite{1980LDY}) are always degenerated
into a five-point stencil which looks like ``+",
while the stencil of the five-point MACH-like scheme looks like ``$\times$".
Therefore, this adds an extra difficulty for error estimates.
The other important work of this paper is that we present the strict theoretical analysis for
the five-point MACH-like scheme.
It is divided into two parts.
On the one hand, we discuss the local truncation error. The main
difficulty results from the discontinuous coefficients of the
interface. By the homogeneous jump conditions on the interface and
Taylor expansions, we derive that the local truncation error of the
interior nodes adjacent to the interface is $O(h)$. Furthermore, we
get the local truncation error $O(h^2)$ under the Assumption {\bf I}
(i.e., we use harmonic average method for the diffusion coefficients
and the second derivative function with respect to the tangent
direction of the exact solution is equal to zero on the interface).
On the other hand, we focus on the global error estimates. Firstly,
the error difference equations are decomposed into two relatively
simple ones. Then, by using the discrete sine transform
(\cite{2005SN,2007MV}) and combining with some analytical
techniques, we convert these 2D difference equations into two kinds
of one-dimensional (1D) difference equations, and the estimates of
these difference equations are deduced. Hence, we demonstrate that
the global error estimation of the five-point scheme is $O(h |\ln
h|)$ in the maximum norm. Furthermore, under the Assumption {\bf I},
we obtain the asymptotic optimal error estimate $O(h^2 |\ln h|)$. In
addition, we investigate the approximation of the five-point
MACH-like scheme by several typical numerical examples. Numerical
results are carried on to confirm the theoretical ones.


The paper is organized as follows. In Section \ref{sec:2}, we describe the construction of the
MACH-like finite volume method on an arbitrary quadrilateral grid, and present a five-point MACH-like scheme for special.
Section \ref{sec:3} presents the local truncation error for the five-point scheme.
Then the global error estimation in the maximum norm is derived in Section \ref{sec:4}.
After that, the accuracy of the five-point MACH-like scheme is verified numerically.
In Section \ref{sec:6} we draw some conclusions on our works.

\section{\label{sec:2}Model problem and finite volume scheme}\setcounter{equation}{0}

We consider the following interface problem
\begin{eqnarray}\label{Benchmark-Eq1}
\left\{\begin{array}{ll}
    -\nabla\cdot(\kappa \nabla u) + u = f,  & \mbox{in}~\Omega, \\ 
    u = 0, & \mbox{on}~\partial \Omega,
\end{array}\right.
\end{eqnarray}
where $\Omega \in \mathbb{R}^2$ is a bounded polygonal domain  with boundary $\partial \Omega$,
$f$ is a given function and the diffusion coefficient $\kappa$ is positive and piecewice
constant on polygonal subdomains of $\Omega$ with possible large discontinuities across subdomain boundaries which is simply interface for short.
Let $\kappa = \kappa_i >0~\mbox{in}~D_i$, for $i=1,2,\cdots,J$. Here $\{D_i\}_{i=1}^J$ is a partition of $\Omega$, where $D_i$
is an open polygonal subdomain.

Denote $\Gamma_{i,j} = \bar{D}_i \cap \bar{D}_j(i\neq j)$ and $\Gamma = \cup_{i,j=1}^J \Gamma_{i,j}$. We assume that
\begin{eqnarray}\label{Benchmark-Eq3}
[u] = [\kappa \frac{\partial u}{\partial \vec{n}}] = 0,~~\mbox{on}~\Gamma,
\end{eqnarray}
where $\vec{n}$ is the unit outward normal vector on $\Gamma$ and $[\zeta](\zeta = u, \kappa \frac{\partial u}{\partial \vec{n}})$
denotes the difference of the  right and left limits of $\zeta$ at any point of $\Gamma$.

Let $\mathcal{Q}_h$ be a structured quadrilateral partition of $\Omega$, and
$
X = \{X_{i,j}=(x_i,y_j),~i=0,1,\cdots,N_x, j=0,1,\cdots,N_y\},
$
where $N_x, N_y$ are given positive integers.

Next, We derive a kind of FVM with vertex unknowns for solving \eqref{Benchmark-Eq1} and \eqref{Benchmark-Eq3},
which is based on the method of Michael H. Frese in \cite{MHF1987}.
We denote it as MACH-like FVM for convenience.

For any given interior node $X_{i,j}, i=1,2,\cdots,N_x-1, j=1,2,\cdots,N_y-1$ (see Fig. \ref{MACH-vol}). We denote
open area $Q_l(l=1,2,3,4)$ as the $l$-th quadrilateral element adjacent to $X_{i,j}$,
where $O_l, l=1,2,3,4$ are the midpoints of the grid sides $X_{i,j-1}X_{i,j}, X_{i+1,j}X_{i,j}, X_{i,j+1}X_{i,j}$ and $X_{i-1,j}X_{i,j}$, respectively.
The quadrangle $X_{i,j}O_{l-1}A_lO_l$ is a parallelogram, where $A_l \in Q_l$.
Thus, the control volume of $X_{i,j}$ is defined as the octagon area which is surrounded by $A_l, O_l(l=1,2,3,4)$. We denote it briefly by $V_{i,j}$.

Fig. \ref{MACH-vol} (a) and (b) show the  typical cases which have an interface in $Q_{l}(l=1,2,3,4)$ or not.

\begin{figure}[t]
\centering
   \includegraphics[scale=0.45]{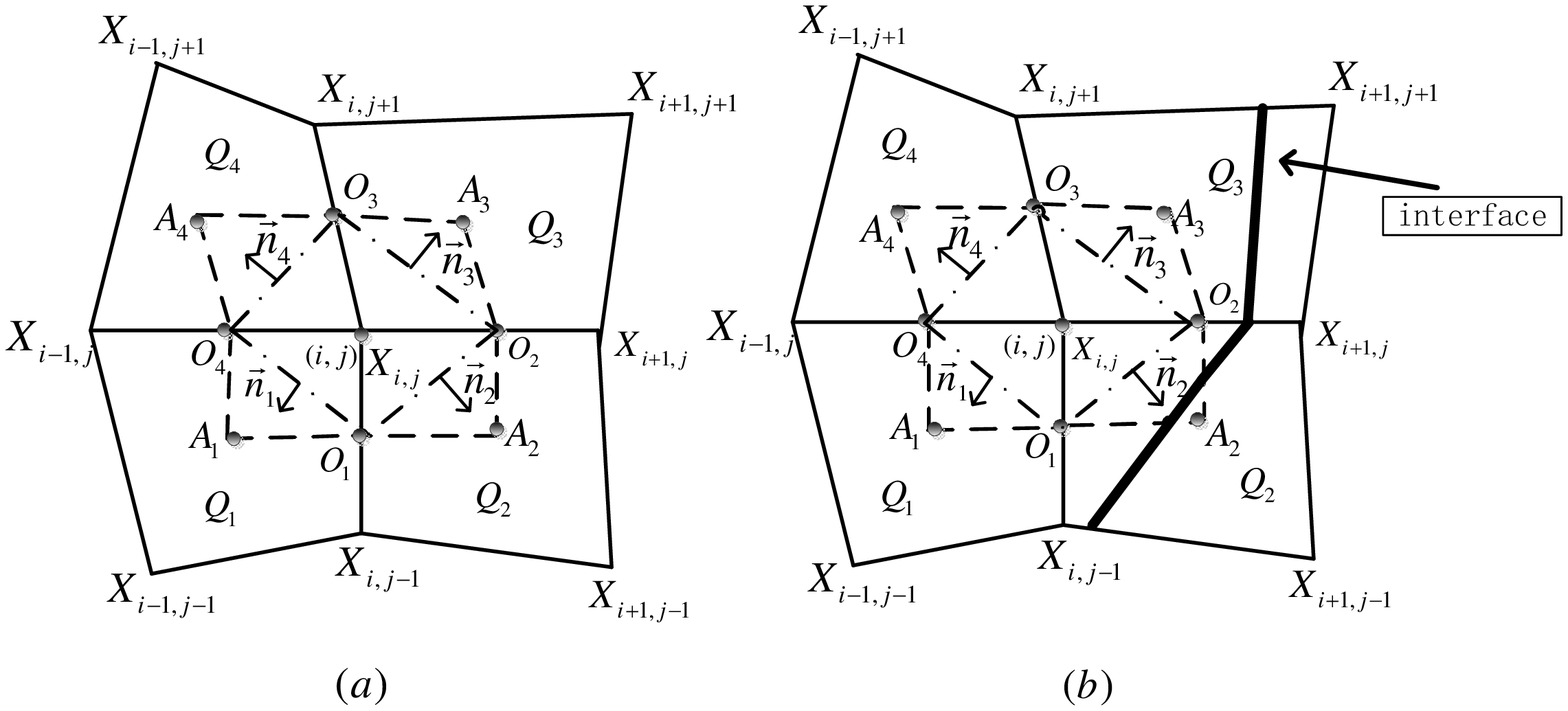}
   \caption{(a) No interface in $Q_l(l=1,2,3,4)$. (b) Only one interface in $Q_2$ and $Q_3$.}
  \label{MACH-vol}
\end{figure}

Integrating \eqref{Benchmark-Eq1} over $V_{i,j}$, the equation \eqref{Benchmark-Eq1} leads to
\begin{equation*}
\int_{V_{i,j}} (-\nabla \cdot (\kappa \nabla u) + u) dX = \int_{V_{i,j}} f dX.
\end{equation*}

By the divergence theorem and the continuity condition of the flux $\kappa \frac{\partial u}{\partial \vec{n}}$, we get
\begin{equation}\label{equation-vol}
- \int_{\partial V_{i,j}} \kappa \frac{\partial u}{\partial \vec{n}} dS + \int_{V_{i,j}} u dX  = \int_{V_{i,j}} f dX,
\end{equation}
where $\partial V_{i,j}$ is the boundary of $V_{i,j}$ and $\vec{n}$ is the unit outward  normal vector of $\partial V_{i,j}$.

Denote $\vec{n}_l$ as the unit outward normal vector on the side $O_{l-1}O_l$ for $l=1,2,3,4$ (see Fig. \ref{MACH-vol}).
Combining \eqref{equation-vol} and the rectangular integral  formula,  we can obtain
\begin{eqnarray}\label{2d-5pt-eq-3}\nonumber
 &&- \kappa_1^* (\nabla u)|_{Q_1} \cdot \vec{n}_1 |O_4O_1|  -  \kappa_2^* (\nabla u)|_{Q_2} \cdot \vec{n}_2 |O_1O_2| \\
 &&  - \kappa_3^* (\nabla u)|_{Q_3} \cdot \vec{n}_3 |O_2O_3|  - \kappa_4^* (\nabla u)|_{Q_4} \cdot \vec{n}_4 |O_3O_4|  + |V_{i,j}| u_{i,j}  \approx \int_{V_{i,j}} f dX,
\end{eqnarray}
where $\kappa_l^*$ is the averaging of $\kappa$ over $Q_l$ and $(\nabla u)|_{Q_l}$ is the integral average value over $Q_l$.

Hence, from \eqref{2d-5pt-eq-3}, we know that it is an urgent task to get the approximate calculation formula of
$(\nabla u)|_{Q_l}$ for $l=1,2,3,4$.

Any open area $Q_l$ is shown in Fig. \ref{MACH-dual-vol}, where
$X_m^l(m=1,2,3,4)$ is the vertex of $Q_l$,
$\vec{n}_m^c(m=1,2,3,4)$ is the unit outward normal vector on the side $X_m^lX_{m+1}^l(m=1,2,3,4,X_5^l = X_1^l)$.
\begin{figure}[t]
\centering
   \includegraphics[scale=0.55]{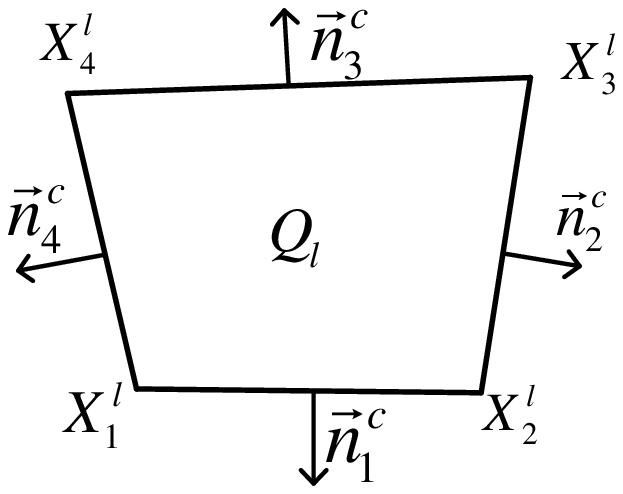}
   \caption{Illustration for $Q_l$. }
  \label{MACH-dual-vol}
\end{figure}

Motivated by the finite volume method in M. H. Frese \cite{MHF1987}, we can obtain
\begin{eqnarray}\label{2d-5pt-eq-6}\nonumber
(\nabla u)|_{Q_l} &\approx& \frac{1}{|Q_l|}[\frac{1}{2}(u(X_1^l) + u(X_2^l))|X_1^lX_2^l| \vec{n}_1^c +  \frac{1}{2}(u(X_2^l) + u(X_3^l))|X_2^lX_3^l| \vec{n}_2^c \\
&& + \frac{1}{2}(u(X_3^l) + u(X_4^l))|X_3^lX_4^l| \vec{n}_3^c + \frac{1}{2}(u(X_4^l) + u(X_1^l))|X_4^lX_1^l| \vec{n}_4^c],
\end{eqnarray}
where $|\zeta|$ is the length or area of $\zeta$.

The option of $\kappa_l^*(l=1,2,3,4)$ is another critical element to construct the MACH-like FVM.
For any given $Q_l(l=1,2,3,4)$, if $Q_l$ is a single material area (see Fig. \ref{MACH-vol}(a) $Q_k(k=1,2,3,4)$ as an example),
without loss of generality, assuming that $Q_l \subset D_m(1\le m\le J)$, then it follows that  $\kappa_l^* = \kappa_m$.
If $Q_l$ is a multiple material area (see, for example, $Q_2, Q_3$ in Fig. \ref{MACH-vol}(b)),
then $\kappa_l^*$ must be obtained from an appropriate average method, such as arithmetic average method, harmonic average method and so on.

Substituting  \eqref{2d-5pt-eq-6} into \eqref{2d-5pt-eq-3}, and combining with the proper $\kappa_l^*(l=1,2,3,4)$, we get the MACH-like FVM of \eqref{Benchmark-Eq1} and \eqref{Benchmark-Eq3} at $X_{i,j}(i=1,2,\cdots,N_x-1, j=1,2,\cdots,N_y-1)$.

\begin{remark}
The MACH-like FVM is always a compact nine-point stencil.
\end{remark}

In particular, we consider a uniform quadrilateral grid in $\Omega = (a,b)\times (c,d)$, where
\begin{eqnarray*}
&& h_x = (b-a)/N_x, ~~h_y = (d-c)/N_y,\\
&& (x_i, y_j) = (a + ih_x,c + jh_y),~i=0,1,\cdots,N_x, j=0,1,\cdots,N_y.
\end{eqnarray*}
The MACH-like FVM at any interior node $X_{i,j}$ express as follows
\begin{eqnarray}\label{dis-9pt-scheme}\nonumber
&&a^{i,j}_{1} u_{i-1,j-1} + a^{i,j}_{2} u_{i,j-1} + a^{i,j}_{3} u_{i+1,j-1} +  a^{i,j}_{4} u_{i-1,j} + a^{i,j}_{5} u_{i,j} \\
&& + a^{i,j}_{6} u_{i+1,j} + a^{i,j}_{7} u_{i-1,j+1} + a^{i,j}_{8} u_{i,j+1} + a^{i,j}_{9} u_{i+1,j+1} = h_x h_y f_{i,j},
\end{eqnarray}
where
\begin{eqnarray}\label{express-aij}
\left\{
\begin{array}{l}
a^{i,j}_1 = -\frac{1}{4}(\frac{h_y}{h_x} + \frac{h_x}{h_y})\kappa_1^*, \\
a^{i,j}_2 = -\frac{1}{4}(-\frac{h_y}{h_x} + \frac{h_x}{h_y})(\kappa_1^* + \kappa_2^*),\\
a^{i,j}_3 = -\frac{1}{4}(\frac{h_y}{h_x} + \frac{h_x}{h_y})\kappa_2^*,\\
a^{i,j}_4 = -\frac{1}{4}(\frac{h_y}{h_x} - \frac{h_x}{h_y})(\kappa_1^* + \kappa_4^*),\\
a^{i,j}_5 = \frac{1}{4}(\frac{h_y}{h_x} + \frac{h_x}{h_y})(\kappa_1^* + \kappa_2^* + \kappa_3^* + \kappa_4^*) + h_x h_y,\\
a^{i,j}_6 = -\frac{1}{4}(\frac{h_y}{h_x} - \frac{h_x}{h_y})(\kappa_2^* + \kappa_3^*), \\
a^{i,j}_7 = -\frac{1}{4}(\frac{h_y}{h_x} + \frac{h_x}{h_y})\kappa_4^*,\\
a^{i,j}_8 = -\frac{1}{4}(-\frac{h_y}{h_x} + \frac{h_x}{h_y})(\kappa_3^* + \kappa_4^*),\\
a^{i,j}_9 = -\frac{1}{4}(\frac{h_y}{h_x} + \frac{h_x}{h_y})\kappa_3^*,\\
f_{i,j}  = \frac{1}{h_x h_y}\int_{V_{i,j}} f dX.
\end{array}
\right.
\end{eqnarray}

Currently, MACH scheme is successfully applied in the numerical simulation of liner implosion system, plasma thrusters and so on (see, e.g. \cite{s-0001,l-0001,m-0001}).
However, there is few error theories about this scheme, and it is always for the local truncation error.
Especially, the strict error theories for the stationary diffusion problems with multiple material cells on Eulerian grids haven't been seen yet.
In the following sections, we will establish a rigorous theoretical analysis of the MACH-like scheme, which is constructed for a simplified model of \eqref{Benchmark-Eq1} and \eqref{Benchmark-Eq3}.

This simplified model only considers two subdomains. Suppose that $\Omega = (0,1)\times (0,1) = D_1 \cup D_2 \cup \Gamma$ and the diffusion coefficient
\begin{eqnarray}\label{def-kappa}
\kappa = \left\{\begin{array}{ll}
\kappa^- \ge 1, & \mbox{in}~D_1,\\
1, & \mbox{in}~D_2,
\end{array}\right.
\end{eqnarray}
where  $D_1 = (0,\frac{1}{2})\times(0,1)$, $D_2 = (\frac{1}{2}, 1)\times(0,1)$ and $\kappa^-$ is a positive constant.

In addition, we assume that $\mathcal{Q}_h$ is a square mesh of $\Omega$, where
$N_x = N_y = N = 2M+1(M \in \mathcal{Z}^+)$,
$h_x = h_y = h$,
introduce the grid nodes and the interior nodes indicated set
$S = \{(i,j), i,j=0,1,\cdots,N\}$ and $S_0 = \{(i,j), i,j=1,2,\cdots,N-1\}$.
Let the indicated set of the interior nodes which are adjacent to the interface or not be
\begin{eqnarray}\label{def-SI-SB}
S_B = S_1^B \cup S_2^B~~and~~S_I = S_1^I \cup S_2^2,
\end{eqnarray}
respectively, where
\begin{eqnarray*}
&& S_1^{I} = \{(i,j), i=1,2,\cdots,M-1, j=1,2,\cdots,N-1\},\\
&& S_2^{I} = \{(i,j), i=M+2,M+3,\cdots,N-1, j=1,2,\cdots,N-1\},\\
&& S_1^{B} = \{(i,j), i=M, j=1,2,\cdots,N-1\},\\
&& S_2^{B} = \{(i,j), i=M+1, j=1,2,\cdots,N-1\}.
\end{eqnarray*}

Denote $\kappa^*$ as the average value of  the diffusion coefficient in the multiple material cells.
From \eqref{dis-9pt-scheme} and \eqref{express-aij}, we can obtain the five-point MACH-like scheme of the simplified model as follows
\begin{eqnarray}\label{2D-MACH-5pt-schemes}
\left\{\begin{array}{ll}
L_h u_{i,j} = h^2 f_{i,j}, &(i,j) \in S_0,\\
u_{0,j} = u_{N, j} = u_{i,0} = u_{i,N}=0, &(i,j) \in S,
\end{array}\right.
\end{eqnarray}
where
\begin{eqnarray} 
\label{elliptic-operator}
L_h u_{i,j} = a_1^{i,j} u_{i-1,j-1} + a_2^{i,j} u_{i+1,j-1} + a_3^{i,j} u_{i,j} + a_4^{i,j} u_{i-1,j+1} + a_5^{i,j} u_{i+1,j+1},~f_{i,j} = f(x_i,y_j),
\end{eqnarray}
and
\begin{eqnarray*}
\left\{
\begin{array}{ll}
a_1^{i,j} = a_2^{i,j} = a_4^{i,j} = a_5^{i,j} = -\frac{1}{2} \kappa^-,~a_3^{i,j} = 2 \kappa^- + h^2,& (i,j) \in S_1^I,\\
a_1^{i,j} = a_4^{i,j} = -\frac{1}{2}\kappa^-,~a_2^{i,j} = a_5^{i,j} = -\frac{1}{2} \kappa^*,~a_3^{i,j} = \kappa^- + \kappa^* + h^2,& (i,j) \in S_1^B,\\
a_1^{i,j} = a_4^{i,j} = -\frac{1}{2}\kappa^*,~a_2^{i,j} = a_5^{i,j} = -\frac{1}{2},~a_3^{i,j} =\kappa^* + 1  + h^2, & (i,j) \in S_2^B,\\
a_1^{i,j} = a_2^{i,j} = a_4^{i,j} = a_5^{i,j} = -\frac{1}{2},~a_3^{i,j} = 2 + h^2,& (i,j)\in S_2^I.
\end{array}\right.
\end{eqnarray*}

The stencil of this five-point scheme which looks like ``$\times$" is shown in Fig. \ref{MACH-5pt} (a).
It is worth pointing out that many frequently-used nine-point FVMs, such as \cite{1980LDY}, are degenerated into a five-point stencil which looks like ``+" (see Fig. \ref{MACH-5pt}(b)).
Hence, this difference will bring some new difficulties to the following error analysis.
\begin{figure}[t]
\centering
   \includegraphics[scale=0.4]{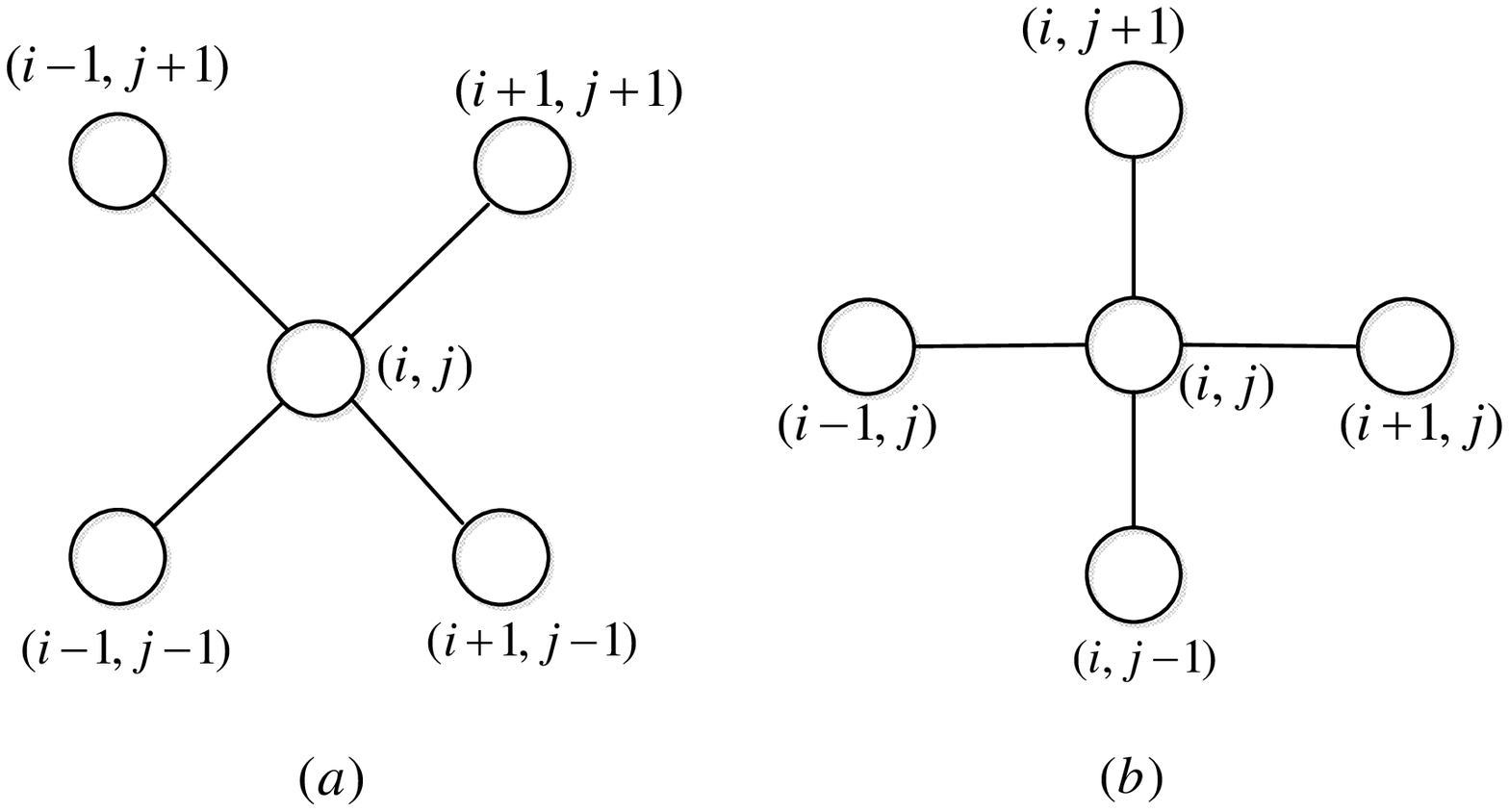}
   \caption{(a) ``$\times$" type five-point stencil. (b) ``+" type five-point stencil.}
  \label{MACH-5pt}
\end{figure}

In the rest of this paper, we will present the rigorous error analysis and corresponding numerical experiments  for the five-point MACH-like scheme \eqref{2D-MACH-5pt-schemes}.

\section{\label{sec:3}Local truncation error estimation}\setcounter{equation}{0}

Let
\begin{eqnarray}\label{def-Rij}
R_{i,j} = L_h u(x_i,y_j) - L_h u_{i,j} = L_h u(x_i,y_j) - h^2 f_{i,j},~~(i,j)\in S_0
\end{eqnarray}
be the local truncation error of \eqref{2D-MACH-5pt-schemes} at $X_{i,j}$,
and denote $x_{i + \frac{1}{2}} = x_{i} + \frac{h}{2}(i=0,1,\cdots,N-1)$, $\delta_{l,m}$ as the Kronecker delta.

Define a function space via
$$W = \{w|~w|_{\bar{D}_i} \in C^6(\bar{D}_i),~i=1,2 \}.$$

In this section, we will present the local truncation error of \eqref{2D-MACH-5pt-schemes}. 

\begin{theorem}\label{theorem-Rij}
If $u \in W$, then the following holds.
\begin{description}
\item[1)] For any interior node which is non adjacent to the interface,  we have
\begin{eqnarray}\label{Rij-equation-A}
R_{i,j} = (\kappa^-)^{\delta_{1,k}} C^{(4)}(x_i,y_j)  h^4 +  O(h^6), ~~(i,j) \in S_k^I, ~k=1,2,
\end{eqnarray}
where
\begin{eqnarray}\label{def-Cij4}
&& C^{(4)}(x,y) = -\frac{1}{12} (\frac{\partial^4 u}{\partial y^4} + \frac{\partial^4 u}{\partial x^4}
+ 6 \frac{\partial^4 u}{\partial y^2 \partial x^2})(x,y).
\end{eqnarray}

\item[2)] For any interior node which is adjacent to the interface, we have
\begin{eqnarray}\label{Rij-equation-B}
R_{i,j} = C_{i}^{(1)}(y_j) h + C_{i}^{(2)}(y_j) h^2
+ C_{i}^{(3)}(y_j) h^3 + O(h^4), ~(i,j) \in S_B,
\end{eqnarray}
where
\begin{eqnarray}\label{def-CM1}
&& C_M^{(1)}(y) = (\kappa^-  - \frac{\kappa^*}{2} - \frac{\kappa^* \kappa^-}{2}) (\frac{\partial u}{\partial x})^-(y),
  ~~C_{M+1}^{(1)}(y) = -C_M^{(1)}(y),\\\label{def-CM2}
&& C_{M}^{(2)}(y) =  \frac{\kappa^*}{8} [\frac{\partial^2 u}{\partial x^2}](y)
+ \frac{1}{2}(\kappa^- (\frac{\partial^2 u}{\partial y^2})^- - \kappa^* (\frac{\partial^2 u}{\partial y^2})^+)(y),\\\label{def-CM+12}
&& C_{M+1}^{(2)}(y) =  - \frac{\kappa^*}{8} [\frac{\partial^2 u}{\partial x^2}](y)
+ \frac{1}{2}((\frac{\partial^2 u}{\partial y^2})^+ - \kappa^* (\frac{\partial^2 u}{\partial y^2})^-)(y),\\\label{def-CM3}
&& C_{M}^{(3)}(y) = (\frac{\kappa^-}{24}  (\frac{\partial^3 u}{\partial x^3})^-
- \frac{\kappa^*}{24} \{\frac{\partial^3 u}{\partial x^3}\}
+ \frac{1}{4} (\kappa^- (\frac{\partial^3 u}{\partial y^2 \partial x})^-
                          - \kappa^*(\frac{\partial^3 u}{\partial y^2 \partial x})^+))(y), \\\label{def-CM+13}
&& C_{M+1}^{(3)}(y) =
(- \frac{1}{24}(\frac{\partial^3 u}{\partial x^3})^+
+ \frac{\kappa^*}{24}\{\frac{\partial^3 u}{\partial x^3}\}
+ \frac{1}{4}(-(\frac{\partial^3 u}{\partial y^2 \partial x})^+
+ \kappa^* (\frac{\partial^3 u}{\partial y^2 \partial x})^-))(y),~~~~
\end{eqnarray}
and $(\zeta)^{\pm}(y) := \lim \limits_{x\rightarrow {x_{M+\frac{1}{2}}}^{\pm}} \zeta(x, y)$, $[\zeta] := (\zeta)^- -(\zeta)^+$,
$\{\zeta\} := \frac{1}{2}((\zeta)^- + (\zeta)^+)$.

\end{description}
\end{theorem}

\begin{proof}
We only prove for $i= M$ and $j=1,2,\cdots,N-1$, the remainder of the argument is analogous to it.

By \eqref{elliptic-operator} and \eqref{def-Rij}, we have
\begin{eqnarray}\label{appendix-Rij-3}\nonumber
R_{M,j}
&=& -\frac{1}{2} \kappa^- [u(x_{M-1},y_{j-1})+ u(x_{M-1},y_{j+1})]  -\frac{1}{2} \kappa^* [u(x_{M+1},y_{j-1}) + u(x_{M+1},y_{j+1})]\\
&& + (\kappa^- + \kappa^* + h^2) u(x_{M},y_{j})  - h^2 f(x_M, y_j).
\end{eqnarray}

Substituting the Taylor series expansions of $u(x_{M\pm 1}, y_{j-1})$ and $u(x_{M\pm 1}, y_{j+1})$ at $(x_{M\pm 1}, y_{j})$ into \eqref{appendix-Rij-3}, we have
\begin{eqnarray}\label{appendix-Rij-3-1}\nonumber
R_{M,j}
&=&  \kappa^- [u(x_M,y_j) - u(x_{M-1},y_{j})] + \kappa^*[u(x_M,y_j) - u(x_{M+1},y_j)] \\\nonumber
&&  -\frac{1}{2}h^2[\kappa^- \frac{\partial^2 u}{\partial y^2}(x_{M-1},y_{j}) + \kappa^* \frac{\partial^2 u}{\partial y^2}(x_{M+1},y_{j})]\\
&& + h^2 u(x_{M},y_{j})  - h^2 f(x_M, y_j) + O(h^4).
\end{eqnarray}

Further, substituting the Taylor series expansions of $u(x_M, y_j)$, $u(x_{M\pm 1}, y_j)$, $f(x_M,y_j)$ and $\frac{\partial^2 u}{\partial y^2}(x_{M\pm 1},y_{j})$ at $(x_{M+\frac{1}{2}},y_j)$
into \eqref{appendix-Rij-3-1}, and using $[u]|_{\Gamma} = 0$,  equation \eqref{Benchmark-Eq1} and $[\kappa \frac{\partial u}{\partial \vec{n}}]|_{\Gamma} = 0$, $R_{M,j}$ is rewritten as
\begin{eqnarray}\nonumber
R_{M,j}
&=& h(\kappa^- - \frac{\kappa^*}{2} - \frac{\kappa^* \kappa^-}{2})(\frac{\partial u}{\partial x})^-(y_j)\\\nonumber
&& + h^2 (\frac{\kappa^*}{8}[\frac{\partial^2 u}{\partial x^2}] + \frac{1}{2}(\kappa^- (\frac{\partial^2 u}{\partial y^2})^- - \kappa^* (\frac{\partial^2 u}{\partial y^2})^+))(y_j) \\\nonumber
&& + h^3 (\frac{\kappa^-}{24}  (\frac{\partial^3 u}{\partial x^3})^-  - \frac{\kappa^*}{24}\{\frac{\partial^3 u}{\partial x^3}\}
+ \frac{1}{4} (\kappa^- (\frac{\partial^3 u}{\partial y^2 \partial x})^- - \kappa^*(\frac{\partial^3 u}{\partial y^2 \partial x})^+) )(y_j) + O(h^4)
\\\label{RMj}
&=:& C_{M}^{(1)}(y_j) h + C_{M}^{(2)}(y_j) h^2 + C_{M}^{(3)}(y_j) h^3 + O(h^4),
\end{eqnarray}
where $C_{M}^{(1)}(y)$, $C_{M}^{(2)}(y)$ and $C_{M}^{(3)}(y)$ are defined by \eqref{def-CM1}, \eqref{def-CM2} and \eqref{def-CM3}, respectively.

This finishes the proof.

\end{proof}

For a lot of diffusion problems with interface (such as \cite{2013KMX}), the changing of the solution along the normal direction (here for the $x$ direction) of the interfaces is far more quickly than the tangential direction (here for the $y$ direction). As a result, we suppose that
\begin{eqnarray}\label{assumption-equ}
\lim\limits_{x \to {x_{M+\frac{1}{2}}}^-} \frac{\partial^2 u}{\partial y^2} = \lim\limits_{x \to {x_{M+\frac{1}{2}}}^+} \frac{\partial^2 u}{\partial y^2} = 0.
\end{eqnarray}

Under the above assumption, we can obtain a corollary from \eqref{Rij-equation-B} as follows.
\begin{corollary}\label{corollary-1}
Let $u\in W$ and suppose that $\kappa^* = \frac{2\kappa^-}{\kappa^- + 1}$ and \eqref{assumption-equ} are hold, then
\begin{eqnarray}\label{corollary-1-eq1}
R_{i,j} = \tilde{C}_{i}^{(2)}(y_j) h^2 + C_{i}^{(3)}(y_j) h^3 + O(h^4),~~(i,j) \in S_B,
\end{eqnarray}
where
\begin{eqnarray}\label{corollary-1-eq2}
&& \tilde{C}_{M}^{(2)}(y) = \frac{\kappa^*}{8} [\frac{\partial^2 u}{\partial x^2}](y)~~and~~\tilde{C}_{M+1}^{(2)}(y) = - \tilde{C}_{M}^{(2)}(y).
\end{eqnarray}
\end{corollary}

For simplicity, we denote the assumptions
$\kappa^* = \frac{2\kappa^-}{\kappa^- + 1}$ and \eqref{assumption-equ} as Assumption~{\bf I}.

\section{\label{sec:4}Global error estimation}\setcounter{equation}{0}

In this section, we will investigate the global error estimates for the five-point MACH-like scheme.

Defined
\begin{eqnarray}\label{def-vec-e}
\vec{e} = (e_{0,0},\cdots,e_{0,N},\cdots,e_{N,0},\cdots,e_{N,N})
\end{eqnarray}
as the error vector, where
\begin{eqnarray}\label{def-eij}
e_{i,j} = u(x_i,y_j) - u_{i,j}.
\end{eqnarray}


Using \eqref{def-vec-e}, \eqref{def-eij} and \eqref{2D-MACH-5pt-schemes}, we can obtain that $\vec{e}$ satisfies the following 2D difference equations 
\begin{equation}\label{eij-equation}
\left\{\begin{array}{ll}
L_h e_{i,j} = R_{i,j}, & (i,j)\in S_0,\\
e_{i,0} = e_{i,N} = e_{0,j} = e_{N,j} = 0, & (i,j) \in S,\\
\end{array}\right.
\end{equation}
where $R_{i,j}$ for $(i,j)\in S_0$ are given by \eqref{Rij-equation-A} and \eqref{Rij-equation-B}.

Let $S_B$ be the indicated set of the interior nodes adjacent to the interface(given by \eqref{def-SI-SB}), denote $e_{i,j} = 0(\forall (i,j)\in S_B)$ in \eqref{eij-equation}, and we have
{\small
\begin{equation}\label{equation-eij1}
\left\{
\begin{array}{ll}
-\frac{1}{2}\kappa^-(e_{i-1,j-1}^{(1)} + e_{i-1,j+1}^{(1)}) + (2\kappa^- + h^2)e_{i,j}^{(1)} - \frac{1}{2}\kappa^-(e_{i+1,j-1}^{(1)} + e_{i+1,j+1}^{(1)}) = R_{i,j}, &(i,j) \in S_1^I,\\
-\frac{1}{2}(e_{i-1,j-1}^{(1)} + e_{i-1,j+1}^{(1)}) + (2 + h^2)e_{i,j}^{(1)} - \frac{1}{2}(e_{i+1,j-1}^{(1)} + e_{i+1,j+1}^{(1)}) = R_{i,j}, & (i,j) \in S_2^I,\\
e_{i,j}^{(1)} = 0, &(i,j) \in S_B,\\
e_{0,j}^{(1)} =  e_{N,j}^{(1)} = e_{i,0}^{(1)} = e_{i,N}^{(1)} = 0, & (i,j)\in S.
\end{array}
\right.
\end{equation}
}

Let $\vec{e}^{(1)} = (e_{0,0}^{(1)},\cdots,e_{0,N}^{(1)},\cdots,e_{N,0}^{(1)},\cdots,e_{N,N}^{(1)})$ be the solution vector of \eqref{equation-eij1}
and the error vector $\vec{e}$ can be decomposed as
\begin{eqnarray}\label{decompose-eij}
\vec{e} = \vec{e}^{(1)} + \vec{e}^{(2)}.
\end{eqnarray}
By \eqref{eij-equation}, \eqref{equation-eij1} and \eqref{decompose-eij}, it is easy to check that $\vec{e}^{(2)} := (e_{0,0}^{(2)},\cdots,e_{0,N}^{(2)},\cdots,e_{N,0}^{(2)},\cdots,e_{N,N}^{(2)})$ satisfies the following difference equations
{\small
\begin{equation}\label{equation-eij2-000}
\left\{
\begin{array}{ll}
-\frac{1}{2}\kappa^-(e_{i-1,j-1}^{(2)} + e_{i-1,j+1}^{(2)}) + (2\kappa^- + h^2)e_{i,j}^{(2)} - \frac{1}{2}\kappa^-(e_{i+1,j-1}^{(2)} + e_{i+1,j+1}^{(2)}) = 0, & (i,j) \in S_1^I,\\
-\frac{1}{2}\kappa^-(e_{i-1,j-1}^{(2)} + e_{i-1,j+1}^{(2)}) + (\kappa^- + \kappa^* + h^2)e_{i,j}^{(2)} - \frac{1}{2}\kappa^*(e_{i+1,j-1}^{(2)} + e_{i+1,j+1}^{(2)}) = \gamma_{M,j}, & (i,j) \in S_1^B,\\
-\frac{1}{2}\kappa^*(e_{i-1,j-1}^{(2)} + e_{i-1,j+1}^{(2)}) + (\kappa^* + 1 + h^2)e_{i,j}^{(2)} - \frac{1}{2}(e_{i+1,j-1}^{(2)} + e_{i+1,j+1}^{(2)}) = \gamma_{M+1,j}, & (i,j) \in S_2^B,\\
-\frac{1}{2}(e_{i-1,j-1}^{(2)} + e_{i-1,j+1}^{(2)}) + (2 + h^2)e_{i,j}^{(2)} - \frac{1}{2}(e_{i+1,j-1}^{(2)} + e_{i+1,j+1}^{(2)}) = 0, & (i,j) \in S_2^I,\\
e_{0,j}^{(2)} = e_{N,j}^{(2)} = e_{i,0}^{(2)} = e_{i,N}^{(2)} = 0, & (i,j)\in S,
\end{array}
\right.
\end{equation}
}
where
\begin{eqnarray*}
\gamma_{M,j}   = R_{M,j} + \frac{1}{2}\kappa^-(e_{M-1,j-1}^{(1)} + e_{M-1,j+1}^{(1)}),~~
\gamma_{M+1,j} = R_{M+1,j} + \frac{1}{2}(e_{M+2,j-1}^{(1)} + e_{M+2,j+1}^{(1)}).
\end{eqnarray*}

Then, we will employ the dimension reduction techniques (convert the 2D problems into the 1D problems) to estimate $\vec{e}^{(l)}(l=1,2)$.
Therefore, for any given $i=1,2,\cdots,N-1$, $l=1,2$, introduce the  discrete sine transform (see, e.g. \cite{2005SN}) for sequence $\{e_{i,j}^{(l)}\}_{j=1}^{N-1}$ and $\{R_{i,j}\}_{j=1}^{N-1}$ as
\begin{eqnarray}\label{expand-baruij-y-e}
\bar{e}_{i,k}^{(l)} = \sqrt{2h} \sum\limits_{j=1}^{N-1} e_{i,j}^{(l)} \sin (j \pi y_k),~\bar{R}_{i,k} = \sqrt{2h} \sum\limits_{j=1}^{N-1} R_{i,j} \sin (j \pi y_k),~~k=1,2,\cdots,N-1.
\end{eqnarray}
Using
$$\sin (k\pi y_j) = \sin (j\pi y_k),~~(k,j)\in S_0,$$
and
$$\sum\limits_{k=1}^{N-1} \sin (k\pi y_l) \sin (k\pi y_m) = \frac{1}{2h} \delta_{l,m},$$
one can easily confirm
\begin{eqnarray}\label{expand-uij-y-e}
e_{i,j}^{(l)} = \sqrt{2h} \sum\limits_{k=1}^{N-1} \bar{e}_{i,k}^{(l)} \sin (k \pi y_j),~R_{i,j} = \sqrt{2h} \sum\limits_{k=1}^{N-1} \bar{R}_{i,k} \sin (k \pi y_j),~~j = 1,2,\cdots,N-1,
\end{eqnarray}
and where the sequences $\{e_{i,j}^{(l)}\}_{j=1}^{N-1}$ and $\{R_{i,j}\}_{j=1}^{N-1}$ are called as the inverse discrete sine transform of $\{\bar{e}_{i,k}^{(l)}\}_{k=1}^{N-1}$ and $\{\bar{R}_{i,k}\}_{k=1}^{N-1}$, respectively.



Let
\begin{eqnarray}\label{def-vector-barel}
\vec{\bar{e}}_k^{(l)} = (\bar{e}_{0,k}^{(l)}, \cdots, \bar{e}_{N,k}^{(l)}),~~k=1,2,\cdots,N-1, l=1,2,
\end{eqnarray}
where $\bar{e}_{i,k}^{(l)}(i=1,2,\cdots,N-1)$ is defined by \eqref{expand-baruij-y-e}, and
\begin{eqnarray}\label{bare-0Nk}
\bar{e}_{0,k}^{(l)} = \bar{e}_{N,k}^{(l)} = 0.
\end{eqnarray}
Using \eqref{expand-baruij-y-e} and \eqref{expand-uij-y-e}, we can convert the estimate problems of $\vec{e}^{(l)}(l=1,2)$ into $\vec{\bar{e}}_k^{(l)}(k=1,2,\cdots,N-1)$.

Denote the maximum norm of any vector $\vec{\zeta}=(\zeta_1,\cdots,\zeta_{n})$ as $\|\vec{\zeta}\|_{\infty}$,
and it is calculated as
$\max\limits_{i} |\zeta_i|$.
We will present the estimation of $\|\vec{\bar{e}}_k^{(1)}\|_{\infty}(k=1,2,\cdots,N-1)$ firstly.
It is easy to verify that $\vec{\bar{e}}_k^{(1)}(k=1,2,\cdots,N-1)$ satisfies the following 1D difference equations
\begin{eqnarray}\label{bareij-equation-result-1}
\left\{\begin{array}{l}
- \kappa^- \bar{e}^{(1)}_{i-1,k} \cos (k\pi h) + (2\kappa^- + h^2) \bar{e}^{(1)}_{i,k} - \kappa^- \bar{e}^{(1)}_{i+1,k}\cos (k\pi h) = \bar{R}_{i,k},~i=1,2,\cdots,M-1,\\
- \bar{e}^{(1)}_{i-1,k} \cos (k\pi h) + (2 + h^2) \bar{e}^{(1)}_{i,k} - \bar{e}^{(1)}_{i+1,k}\cos (k\pi h) = \bar{R}_{i,k},~i=M+2,M+3,\cdots,N-1,\\
\bar{e}^{(1)}_{M,k} =  \bar{e}^{(1)}_{M+1,k} = 0,~~\bar{e}^{(1)}_{0,k} = \bar{e}^{(1)}_{N,k} = 0.
\end{array}\right.~~
\end{eqnarray}

As a matter of fact, using \eqref{expand-baruij-y-e} and noting that $e_{i,j}^{(1)} = 0(\forall (i,j) \in S_B)$, we can get
$\bar{e}^{(1)}_{M,k} =  \bar{e}^{(1)}_{M+1,k} = 0$. Hence, from \eqref{bare-0Nk}, it follows that we only need to demonstrate the former $N-3$ equations of \eqref{bareij-equation-result-1}
are established.

Substituting \eqref{expand-uij-y-e} into the part of \eqref{equation-eij1} where $(i,j) \in S_1^I$, we get
{\small
\begin{eqnarray}\label{bareij-equation-1}\nonumber
&&-\frac{1}{2}\kappa^-\sqrt{2h} \sum\limits_{k=1}^{N-1} \bar{e}_{i-1,k}^{(1)} [\sin (k \pi y_{j-1}) + \sin (k \pi y_{j+1})] + (2\kappa^- + h^2)\sqrt{2h} \sum\limits_{k=1}^{N-1} \bar{e}_{i,k}^{(1)}\sin (k \pi y_{j})\\
&& - \frac{1}{2}\kappa^-\sqrt{2h} \sum\limits_{k=1}^{N-1} \bar{e}_{i+1,k}^{(1)} [\sin (k \pi y_{j-1}) + \sin (k \pi y_{j+1})] = \sqrt{2h} \sum\limits_{k=1}^{N-1} \bar{R}_{i,k}\sin (k \pi y_{j}),~~j=1,2,\cdots,N-1.~~~~~~~~
\end{eqnarray}
}

Owing to the characterization property of the discrete sine function $\sin (k\pi y_j)(k,j=1,2,\cdots,N)$ that
\begin{eqnarray*}
\sin (k \pi y_{j-1}) + \sin (k \pi y_{j+1})
= 2\sin (k\pi y_j) \cos (k\pi h),
\end{eqnarray*}
equation \eqref{bareij-equation-1} can be further represented as
\begin{eqnarray*}
&&\sqrt{2h} \sum\limits_{k=1}^{N-1}[- \kappa^- \bar{e}_{i-1,k}^{(1)} \cos (k\pi h) + (2\kappa^- + h^2) \bar{e}_{i,k}^{(1)} - \kappa^- \bar{e}_{i+1,k}^{(1)}\cos (k\pi h) - \bar{R}_{i,k}] \sin (k\pi y_j) = 0,\\
&&~~j=1,2,\cdots,N-1.
\end{eqnarray*}
Therefore, we have
\begin{eqnarray*}
&& - \kappa^- \bar{e}_{i-1,k}^{(1)} \cos (k\pi h) + (2\kappa^- + h^2) \bar{e}_{i,k}^{(1)} - \kappa^- \bar{e}_{i+1,k}^{(1)}\cos (k\pi h) - \bar{R}_{i,k} = 0,~~k = 1,2,\cdots,N-1.~~~~~~~~
\end{eqnarray*}
Namely, the former $M-1$ equations of \eqref{bareij-equation-result-1} are established.

The rest of the proof is similar to before. The proof is completed.

In order to estimate $\|\vec{\bar{e}}_k^{(1)}\|_{\infty}(k=1,2,\cdots,N-1)$ by using the difference equations \eqref{bareij-equation-result-1},
we need to introduce several lemmas as follows.
\begin{lemma}\label{lemma-4}
Let $\varphi_j := \varphi(y_j)(j=1,2,\cdots,N-1)$, and assume that $\varphi(y) \in C^2([0,1])$, then
\begin{eqnarray}\label{barRkj-000}
|\sum\limits_{j=1}^{N-1} \varphi_j \sin (j\pi y_k)| &\lesssim& \frac{\|\varphi\|_{2,\infty}}{\sin \frac{k\pi h}{2}}.
\end{eqnarray}
\end{lemma}
\begin{proof}
Since
\begin{eqnarray*}
 \sum\limits_{j=2}^{N-2} (\varphi_{j-1} + \varphi_{j+1}) \sin (j \pi y_k)
&=& \sum\limits_{j=1}^{N-3} \varphi_j \sin ((j+1) \pi y_k) + \sum\limits_{j=3}^{N-1} \varphi_j \sin ((j-1) \pi y_k) \\
 &=& 2 \cos(\pi y_k) \sum\limits_{j=1}^{N-1} \varphi_j \sin (j\pi y_k) -  \sin (\pi y_k)((-1)^{k-1} \varphi_{N-2}  -  \varphi_{2}),
\end{eqnarray*}
we have
\begin{eqnarray*}
(\cos(\pi y_k) - 1)\sum\limits_{j=1}^{N-1} \varphi_j \sin (j\pi y_k) &=& \sum\limits_{j=2}^{N-2} (\frac{\varphi_{j-1} + \varphi_{j+1}}{2}  - \varphi_j)\sin (j \pi y_k) \\
&&  +  \sin (\pi y_k)((-1)^{k-1}(\frac{1}{2} \varphi_{N-2} - \varphi_{N-1})
 +   (\frac{1}{2}\varphi_{2} - \varphi_{1})).
\end{eqnarray*}
As this result, and noting that $\cos(\pi y_k) - 1 = -2\sin^2 \frac{k\pi h}{2} \neq 0(k=1,2,\cdots,N-1)$, one we can obtain
\begin{eqnarray}\label{equation-fi}\nonumber
\sum\limits_{j=1}^{N-1} \varphi_j \sin (j\pi y_k)
&=& - \frac{1}{2\sin^2 \frac{k\pi h}{2}} \sum\limits_{j=2}^{N-2}(\frac{\varphi_{j-1} + \varphi_{j+1}}{2} - \varphi_j)\sin (j\pi y_k)\\
&& + \frac{\cos \frac{k\pi h}{2}}{\sin \frac{k\pi h}{2}} (-(\frac{1}{2}\varphi_{2} - \varphi_{1}) + (-1)^{k} (\frac{1}{2}\varphi_{N-2}  - \varphi_{N-1})).
\end{eqnarray}

Using
\begin{eqnarray}\label{sin-inequiv-0}
\sin x \ge \frac{2}{\pi}x, ~\forall x\in (0,\frac{\pi}{2}),
\end{eqnarray}
we get
\begin{eqnarray}\label{inequiv-sin}
\sin \frac{k\pi h}{2} \ge k h > 0.
\end{eqnarray}

From \eqref{equation-fi} and \eqref{inequiv-sin}, it follows
\begin{eqnarray}\label{equation-fi-000}\nonumber
|\sum\limits_{j=1}^{N-1} \varphi_j \sin (j\pi y_k)|
&\lesssim& \frac{h^{-1}}{k \sin \frac{k\pi h}{2}} \sum\limits_{j=2}^{N-2}|\frac{\varphi_{j-1} + \varphi_{j+1}}{2} - \varphi_j|\\\nonumber
&& + \frac{1}{\sin \frac{k\pi h}{2}} (|\varphi_{2}| + |\varphi_{1}| + |\varphi_{N-2}| + |\varphi_{N-1}|)\\
&\lesssim& \frac{h^{-1}}{k \sin \frac{k\pi h}{2}} \sum\limits_{j=2}^{N-2}|\frac{\varphi_{j-1} + \varphi_{j+1}}{2} - \varphi_j|  + \frac{\|\varphi\|_{\infty}}{\sin \frac{k\pi h}{2}}.
\end{eqnarray}

Noting that $\varphi(y) \in C^2([0,1])$ and by the Taylor series expansions of $\varphi_{j\pm 1}$ at $y=y_j$, we have
\begin{eqnarray}\label{proof-Rik-1-000}
|\frac{\varphi_{j-1} + \varphi_{j+1}}{2} - \varphi_j| \lesssim h^2 \|\varphi''\|_{\infty},~~j=2,3,\cdots,N-2.
\end{eqnarray}

Using \eqref{equation-fi-000}, \eqref{proof-Rik-1-000} and paying attention to $N = h^{-1}$, \eqref{barRkj-000} holds.
\end{proof}

Set the integer $K > 3$, $\beta \in \mathbb{R}$, a $K \times K$ square matrix is defined by
\begin{eqnarray}\label{T-form}
T = \left(\begin{array}{cccc}
\beta & -1     &        &       \\
-1    & \beta  & -1     &       \\
      & \ddots & \ddots &\ddots \\
      &        &  -1   &\beta\\
\end{array}\right).
\end{eqnarray}
For any given real number $\lambda$, let
\begin{eqnarray}\label{Z-j}
 Z_j(\lambda) &=& \lambda^j - \lambda^{-j},~j=0,1,2,\cdots.
\end{eqnarray}
If $|\lambda|>1$, paying attention to $Z_0(\lambda) = 0$ and $|Z_j(\lambda)| = |\lambda|^j - |\lambda|^{-j}$ , we have
\begin{eqnarray}\label{Zj-property-000}
\frac{|Z_j(\lambda)|}{|Z_i(\lambda)|} < 1,~~1 \le j < i.
\end{eqnarray}

The following lemma presents the estimation of the elements of $T^{-1}$, where $T^{-1}$ is the inverse of $T$.

\begin{lemma}\label{lemma-3}
If $|\beta| > 2$, $T$ is defined by \eqref{T-form} and $T^{-1} := (\tilde{t}_{i,j})_{K\times K}$, then
\begin{eqnarray}\label{tij-property-2-2}
& |\tilde{t}_{i,i}| = \max\limits_{j} |\tilde{t}_{i,j}|, &i=1,2,\cdots,K,
\end{eqnarray}
and
\begin{eqnarray}\label{tij-property-t11-tN-1}
&& \tilde{t}_{K,K} < 1.
\end{eqnarray}
\end{lemma}

\begin{proof}
Since $T T^{-1} = I$, for any given $i=1,2,\cdots,K$, it follows
\begin{eqnarray}\label{lemma-3-eq-1-T}
\left\{\begin{array}{ll}
\beta \tilde{t}_{i,1} - \tilde{t}_{i,2} = \delta_{i,1}, &\\
-\tilde{t}_{i,j-1} + \beta \tilde{t}_{i,j} - \tilde{t}_{i,j+1} = \delta_{i,j}, & j=2,3,\cdots,K-1,\\
-\tilde{t}_{i,K-1} + \beta \tilde{t}_{i,K} = \delta_{i,K}. &
\end{array}\right.
\end{eqnarray}

Noticing that the characteristic equation of \eqref{lemma-3-eq-1-T} is
\begin{eqnarray}\label{characteristic-equation-beta}
\lambda^2 -\beta \lambda + 1 = 0.
\end{eqnarray}
Using $|\beta|>2$, we get the characteristic root of equation \eqref{characteristic-equation-beta} with its absolute value greater than 1 as follows
\begin{eqnarray}\label{lemma-3-eq-5-solution-1}
 \lambda_{\beta} = sgn(\beta)\left(\frac{|\beta|}{2} + \sqrt{(\frac{\beta}{2})^2 - 1}\right),
\end{eqnarray}
where $sgn(\cdot)$ is the sign function.

By \eqref{lemma-3-eq-5-solution-1}, we can easily confirm that the component solution of difference equations \eqref{lemma-3-eq-1-T} is
\begin{eqnarray}\label{tij-eq-3}
\tilde{t}_{i,j} = \left\{\begin{array}{ll}
\frac{Z_{K-i+1}(\lambda_\beta) Z_j(\lambda_\beta)} {Z_{K+1}(\lambda_\beta) Z_1(\lambda_\beta)}, & \mbox{if}~j=1,2,\cdots,i,\\
\frac{Z_i(\lambda_\beta) Z_{K-j+1}(\lambda_\beta)} {Z_{K+1}(\lambda_\beta) Z_1(\lambda_\beta)}, & \mbox{if}~ j=i+1,i+2,\cdots,K.
\end{array}\right.
\end{eqnarray}

Hence, using \eqref{tij-eq-3}, we get
\begin{eqnarray}\label{23-tij-eq-3}
\frac{|\tilde{t}_{i,j}|}{|\tilde{t}_{i,i}|}
=
\left\{\begin{array}{ll}
\frac{|Z_j(\lambda_\beta)|} {|Z_i(\lambda_\beta)|}, & \mbox{if}~j=1,2,\cdots,i,\\
\frac{|Z_{K-j+1}(\lambda_\beta)|} {|Z_{K-i+1}(\lambda_\beta)|}, & \mbox{if}~ j=i+1,i+2,\cdots,K.
\end{array}\right.
\end{eqnarray}
Noticing that $|\lambda_{\beta}|>1$, from \eqref{Zj-property-000} and \eqref{23-tij-eq-3}, we have \eqref{tij-property-2-2}.

By using \eqref{Z-j}, \eqref{tij-eq-3} and \eqref{lemma-3-eq-5-solution-1}, we have
\begin{eqnarray}\label{tij-property-2-2-000}
\beta \tilde{t}_{i,i}
= \beta sgn(\lambda_{\beta}) \frac{|Z_{K-i+1}(\lambda_\beta)| |Z_i(\lambda_\beta)|} {|Z_{K+1}(\lambda_\beta)| |Z_1(\lambda_\beta)|} > 0.
\end{eqnarray}
From \eqref{lemma-3-eq-1-T}, \eqref{tij-property-2-2}, \eqref{tij-property-2-2-000} and $|\beta|>2$, then we have
\begin{eqnarray*}\nonumber
1 = -\tilde{t}_{K,K-1} + \beta \tilde{t}_{K,K} \ge -|\tilde{t}_{K,K}| + \beta \tilde{t}_{K,K}
  \ge (|\beta|-1) |\tilde{t}_{K,K}| >  |\tilde{t}_{K,K}|.
\end{eqnarray*}
That is to say  \eqref{tij-property-t11-tN-1} holds.

\end{proof}

For any given vector $\vec{\alpha} = (\alpha_1,\cdots,\alpha_{K})^T \in \mathbb{R}^{K}$, we consider
\begin{eqnarray}\label{lemma-3-eq-1}
T \vec{s} = \vec{\alpha},
\end{eqnarray}
where $T$ is defined by \eqref{T-form} and $\vec{s} = (s_1,\cdots,s_{K})^T$.
\begin{corollary}\label{corollary-2}
Under the assumptions as Lemma \ref{lemma-3}, we have the last component of the solution vector of \eqref{lemma-3-eq-1} satisfies
\begin{eqnarray}\label{lemma-3-eq-4}
|s_{K}| < K \|\vec{\alpha}\|_{\infty}.
\end{eqnarray}
\end{corollary}

\begin{proof}
Using  $\vec{s} = T^{-1}\vec{\alpha}$, \eqref{tij-property-2-2} and \eqref{tij-property-t11-tN-1}, we have
\begin{eqnarray*}
|s_{K}| = |\sum\limits_{j=1}^{K} \tilde{t}_{K,j} \alpha_j|
\le   \|\vec{\alpha}\|_{\infty} \sum\limits_{j=1}^{K} |\tilde{t}_{K,j}|
\le   \|\vec{\alpha}\|_{\infty} K |\tilde{t}_{K,K}|
<     \|\vec{\alpha}\|_{\infty} K,
\end{eqnarray*}
which completes the proof of \eqref{lemma-3-eq-4}.
\end{proof}

Using the above lemmas and corollary, we can get the estimation of $\|\vec{\bar{e}}_k^{(1)}\|_{\infty}(k=1,2,\cdots,N-1)$.
\begin{theorem}\label{theorem-1}
If $u \in W$, then
\begin{eqnarray}\label{lemma-3-eq-3-e-1-result}
\|\vec{\bar{e}}_{k}^{(1)}\|_{\infty}
\lesssim \frac{h^{\frac{3}{2}}}{k},~~k=1,2,\cdots, N-1,
\end{eqnarray}
and
\begin{eqnarray}\label{lemma-3-eq-4-e-2}
 |\bar{e}^{(1)}_{M-1,k}| \lesssim \frac{h^{\frac{7}{2}}}{\sin \frac{k\pi h}{2}|\cos (k\pi h)|},~~
 |\bar{e}^{(1)}_{M+2,k}| \lesssim \frac{h^{\frac{7}{2}}}{\sin \frac{k\pi h}{2}|\cos (k\pi h)|}.
\end{eqnarray}
\end{theorem}
\begin{proof}
Noticing that the inequation of \eqref{lemma-3-eq-3-e-1-result} holds  if and only if
\begin{eqnarray}\label{lemma-3-eq-3-e-1}
\max\limits_{0\le i\le M}|\bar{e}_{i,k}^{(1)}|
 \lesssim \frac{h^{\frac{3}{2}}}{k}~\mbox{and}~
\max\limits_{M+1\le i\le N}|\bar{e}_{i,k}^{(1)}|
\lesssim \frac{h^{\frac{3}{2}}}{k}.
\end{eqnarray}

Now we turn to prove the first inequation of \eqref{lemma-3-eq-4-e-2} and \eqref{lemma-3-eq-3-e-1}, respectively, the second inequation of \eqref{lemma-3-eq-4-e-2} and \eqref{lemma-3-eq-3-e-1}
can be proved similarly.

For any given $k=1,2,\cdots,N-1$ and using \eqref{bareij-equation-result-1}, one can obtain that $\{\bar{e}_{i,k}^{(1)}\}_{i=0}^{M}$ satisfies the following difference equations
\begin{eqnarray}\label{bareij-equation-result-1-01-1}
\left\{\begin{array}{ll}
\bar{L}_h\bar{e}^{(1)}_{i,k} = \bar{R}_{i,k},&i=1,2,\cdots,M-1,\\
\bar{e}^{(1)}_{0,k}=\bar{e}^{(1)}_{M,k} = 0,
\end{array}\right.
\end{eqnarray}
where $\bar{L}_h\bar{e}^{(1)}_{i,k} := - \kappa^- \bar{e}^{(1)}_{i-1,k} \cos (k\pi h) + (2\kappa^- + h^2) \bar{e}^{(1)}_{i,k} - \kappa^- \bar{e}^{(1)}_{i+1,k}\cos (k\pi h)$.

Comparing \eqref{bareij-equation-result-1-01-1} with \eqref{lemma-3-eq-1}, and using Corollary \ref{corollary-2}, where $s_i := \kappa^- \bar{e}^{(1)}_{i,k} \cos (k\pi h)$, $|\beta| := |\frac{2\kappa^- + h^2}{\kappa^- \cos (k\pi h)}| > 2$, $\alpha_i := \bar{R}_{i,k}$ and $K:=M-1$, we get
\begin{eqnarray*}
|\kappa^- \bar{e}_{M-1,k} \cos(k\pi h)| = |s_{K}|  < K \|\vec{\alpha}\|_{\infty} < M \max\limits_{1\le i\le M-1}|\bar{R}_{i,k}|.
\end{eqnarray*}
Hence, by using $\kappa^- \ge 1$ (see \eqref{def-kappa}), we have
\begin{eqnarray}\label{lemma-5-3-eq-a}
|\bar{e}_{M-1,k}| < \frac{ M \max\limits_{1\le i\le M-1}|\bar{R}_{i,k}|}{\kappa^- |\cos (k\pi h)|}.
\end{eqnarray}

Noticing that $u\in W$, from \eqref{expand-baruij-y-e} and \eqref{Rij-equation-A}, we have
\begin{eqnarray}\label{barRik-estimate}
|\bar{R}_{i,k}|
\lesssim h^{\frac{9}{2}} |\sum \limits_{j=1}^{N-1} C^{(4)}(x_i,y_j)\sin (j \pi y_k)| + h^{\frac{11}{2}},~i=1,2,\cdots,M-1.
\end{eqnarray}
Using \eqref{def-Cij4}  and Lemma \ref{lemma-4},
we get
\begin{eqnarray*}
|\sum \limits_{j=1}^{N-1} C^{(4)}(x_i,y_j)\sin (j \pi y_k)| \lesssim \frac{1}{\sin \frac{k\pi h}{2}},~~i=1,2,\cdots,M-1.
\end{eqnarray*}
Substituting this into \eqref{barRik-estimate}, we can obtain
\begin{eqnarray}\label{inequiv-barRik}
|\bar{R}_{i,k}| \lesssim \frac{h^{\frac{9}{2}}}{\sin \frac{k\pi h}{2}},~i=1,2,\cdots,M-1.
\end{eqnarray}

Using \eqref{lemma-5-3-eq-a} and \eqref{inequiv-barRik}, we have the first inequation of  \eqref{lemma-3-eq-4-e-2} is established.

Let
\begin{eqnarray}\label{def-Ej}
E_j = (1 +x_j) h^{-2} \max\limits_{1\le i\le M-1}|\bar{R}_{i,k}|, ~~j=0,1,\cdots,M.
\end{eqnarray}
Noticing that $\bar{L}_h (1+x_j) > h^2$, together with \eqref{def-Ej} and \eqref{bareij-equation-result-1-01-1}, we obtain
\begin{eqnarray}\label{lemma-5-3-eq-b}
\bar{L}_h E_j  >  \max\limits_{1\le i\le M-1}|\bar{R}_{i,k}|
\ge |\bar{L}_h \bar{e}_{j,k}^{(1)}|,~~j=1,2,\cdots,M-1.
\end{eqnarray}

Thus, using \eqref{lemma-5-3-eq-b}, and paying attention to  $E_0 > |\bar{e}^{(1)}_{0,k}|, E_M > |\bar{e}^{(1)}_{M,k}|$, by comparison theorem, we have
\begin{eqnarray}\label{barei-Ei-result}
|\bar{e}^{(1)}_{j,k}| < E_j  \lesssim h^{-2}\max\limits_{1\le i\le M-1}|\bar{R}_{i,k}|,~~j=1,2,\cdots,M-1.
\end{eqnarray}

Hence, from \eqref{barei-Ei-result}, \eqref{inequiv-barRik} and \eqref{inequiv-sin}, the first inequation of \eqref{lemma-3-eq-3-e-1} holds.

The proof is completed.

\end{proof}

Then, we will estimate $\|\vec{\bar{e}}_k^{(2)}\|_{\infty}(k=1,2,\cdots,N-1)$.

Using \eqref{equation-eij2-000},  \eqref{expand-baruij-y-e} and \eqref{expand-uij-y-e}, similar to the derivation process  of \eqref{bareij-equation-result-1}, we can obtain the 1D difference equations that $\vec{\bar{e}}_k^{(2)}(k=1,2,\cdots,N-1)$ satisfies as follows
{\small
\begin{eqnarray}\label{bareij-equation-result-2}
\left\{\begin{array}{l}
- \kappa^- \bar{e}^{(2)}_{i-1,k} \cos (k\pi h) + (2\kappa^- + h^2) \bar{e}^{(2)}_{i,k} - \kappa^- \bar{e}^{(2)}_{i+1,k}\cos (k\pi h) = 0,~~i=1,2,\cdots,M-1,\\
- \bar{e}^{(2)}_{i-1,k} \cos (k\pi h) + (2 + h^2) \bar{e}^{(2)}_{i,k} - \bar{e}^{(2)}_{i+1,k}\cos (k\pi h) = 0,~~i=M+2,M+3,\cdots,N-1,\\
- \kappa^- \bar{e}^{(2)}_{M-1,k} \cos (k\pi h) + (\kappa^- + \kappa^* + h^2) \bar{e}^{(2)}_{M,k} - \kappa^* \bar{e}_{M+1,k}^{(2)}\cos (k\pi h)
= \bar{\gamma}_{M,k}, \\
- \kappa^* \bar{e}^{(2)}_{M,k} \cos (k\pi h) + (\kappa^* + 1 + h^2) \bar{e}^{(2)}_{M+1,k} -  \bar{e}^{(2)}_{M+2,k}\cos (k\pi h)
= \bar{\gamma}_{M+1,k},\\
\bar{e}^{(2)}_{0,k} = \bar{e}^{(2)}_{N,k} = 0,
\end{array}\right.
\end{eqnarray}}
where
\begin{eqnarray}\label{def-gamma-m+1k}
  \bar{\gamma}_{M,k} = \bar{R}_{M,k} + \kappa^- \bar{e}^{(1)}_{M-1,k} \cos (k\pi h),~~
  \bar{\gamma}_{M+1,k} = \bar{R}_{M+1,k} +  \bar{e}^{(1)}_{M+2,k}\cos (k\pi h).
\end{eqnarray}


\begin{lemma}\label{lemma-bare-2}
If function $Z_j$ is defined by \eqref{Z-j}, then, for any given $k=1,2,\cdots, N-1$,  $\vec{\bar{e}}_k^{(2)}$ satisfies
\begin{eqnarray}\label{express-bare2-Mk}
\bar{e}^{(2)}_{M,k}
= \frac{\bar{\gamma}_{M,k}\delta_k(1) + \bar{\gamma}_{M+1,k}\kappa^*}{(\delta_k(\kappa^-)\delta_k(1) - (\kappa^*)^2) \cos(k\pi h)},~~
\bar{e}^{(2)}_{M+1,k}
= \frac{\bar{\gamma}_{M+1,k} + \kappa^* \cos (k\pi h) \bar{e}^{(2)}_{M,k}}{\delta_k(1)\cos (k\pi h)},
\end{eqnarray}
and
\begin{equation}\label{express-bare2-result}
\bar{e}^{(2)}_{i,k} = \left\{\begin{array}{ll}
\frac{Z_i(\lambda_k(\kappa^-))}{Z_M(\lambda_k(\kappa^-))} \bar{e}^{(2)}_{M,k}, & \mbox{if}~i=0,1,\cdots,M-1,\\
\frac{Z_{N-i}(\lambda_k(1))}{Z_M(\lambda_k(1))}\bar{e}^{(2)}_{M+1,k}, & \mbox{if}~i=M+2,M+3,\cdots,N,
\end{array}\right.
\end{equation}
where ($\omega = \kappa^-, 1$)
\begin{eqnarray}\label{def-lambda-function}
&& \lambda_k(\omega) =  sgn(\cos (k\pi h)) \left(\frac{1}{|\cos (k\pi h)|}(1 + \frac{h^2}{2\omega}) + \sqrt{\frac{1}{\cos^2 (k\pi h)}(1 + \frac{h^2}{2\omega})^2 - 1}\right),\\
\label{def-delta-function}
&& \delta_k(\omega) = \Theta_k(\omega)  + \frac{\kappa^* + h^2}{\cos (k\pi h)},~~\Theta_k(\omega) = \omega(\frac{1}{\cos (k\pi h)} - \frac{Z_{M-1}(\lambda_k(\omega))}{Z_M(\lambda_k(\omega))}).
\end{eqnarray}
\end{lemma}

\begin{proof}
Using the former $2M-2$ equations of \eqref{bareij-equation-result-2}, and $\bar{e}^{(2)}_{0,k} = \bar{e}^{(2)}_{N,k} = 0$, we have
\begin{eqnarray}\label{bareij-equation-result-2-111}
\left\{\begin{array}{l}
(2\kappa^- + h^2) \bar{e}^{(2)}_{1,k} - \kappa^- \bar{e}^{(2)}_{2,k}\cos (k\pi h) = 0,\\
- \kappa^- \bar{e}^{(2)}_{i-1,k} \cos (k\pi h) + (2\kappa^- + h^2) \bar{e}^{(2)}_{i,k} - \kappa^- \bar{e}^{(2)}_{i+1,k}\cos (k\pi h) = 0,~i=2,3,\cdots,M-2,~~~~\\
- \kappa^- \bar{e}^{(2)}_{M-2,k} \cos (k\pi h) + (2\kappa^- + h^2) \bar{e}^{(2)}_{M-1,k} = \kappa^- \bar{e}^{(2)}_{M,k}\cos (k\pi h),
\end{array}\right.
\end{eqnarray}
and
\begin{eqnarray}\label{bareij-equation-result-2-002}
\left\{\begin{array}{l}
(2 + h^2) \bar{e}^{(2)}_{M+2,k} - \bar{e}^{(2)}_{M+3,k}\cos (k\pi h) = \bar{e}^{(2)}_{M+1,k} \cos (k\pi h),\\
- \bar{e}^{(2)}_{i-1,k} \cos (k\pi h) + (2 + h^2) \bar{e}^{(2)}_{i,k} - \bar{e}^{(2)}_{i+1,k}\cos (k\pi h) = 0,~~i=M+3,M+4,\cdots,N-2,\\
- \bar{e}^{(2)}_{N-2,k} \cos (k\pi h) + (2 + h^2) \bar{e}^{(2)}_{N-1,k} = 0.
\end{array}\right.
\end{eqnarray}

By using \eqref{bareij-equation-result-2-111}, similar to the derivation process of \eqref{tij-eq-3}($\beta:=\frac{2\kappa^- + h^2}{\kappa^- \cos (k\pi h)}$, $\lambda_{\beta}:=\lambda_k(\kappa^-)$), we have
\begin{equation*}
\bar{e}^{(2)}_{i,k} =
\frac{Z_i(\lambda_k(\kappa^-))}{Z_M(\lambda_k(\kappa^-))} \bar{e}^{(2)}_{M,k}, ~~i=0,1,\cdots,M-1.
\end{equation*}
Similarly, by \eqref{bareij-equation-result-2-002}, we have
\begin{equation*}
\bar{e}^{(2)}_{i,k} =
\frac{Z_{N-i}(\lambda_k(1))}{Z_M(\lambda_k(1))}\bar{e}^{(2)}_{M+1,k}, ~~i=M+2,M+3,\cdots,N,
\end{equation*}
which completes the proof of \eqref{express-bare2-result}.

Then using the $2M-1$-th and $2M$-th equation of \eqref{express-bare2-result}, one can obtain
\begin{eqnarray*}
\left\{\begin{array}{ll}
\delta_k(\kappa^-) \cos (k\pi h) \bar{e}^{(2)}_{M,k} -\kappa^* \cos (k\pi h) \bar{e}_{M+1,k}^{(2)} = \bar{\gamma}_{M,k}, \\
-\kappa^* \cos (k\pi h) \bar{e}^{(2)}_{M,k} + \delta_k(1) \cos (k\pi h) \bar{e}^{(2)}_{M+1,k} = \bar{\gamma}_{M+1,k},
\end{array}\right.
\end{eqnarray*}
where $\delta_k(\omega)$ is defined by \eqref{def-delta-function}.

Solving the above equations, we deduce that $\bar{e}^{(2)}_{M,k}$ and $\bar{e}^{(2)}_{M+1,k}$ satisfy \eqref{express-bare2-Mk}.
\end{proof}

In order to estimate $\vec{\bar{e}}_k^{(2)}$ by \eqref{express-bare2-Mk} and \eqref{express-bare2-result}, we introduce the following two lemmas.

\begin{lemma}\label{delta11-delta22-lemma}
\begin{eqnarray}\label{delta11-delta22}
&\delta_k(\kappa^-) \ge \delta_k(1) > \kappa^* + kh,&~~\mbox{if}~k=1,2,\cdots,M,\\\label{delta11-delta22-000}
&-\delta_k(\kappa^-) \ge  -\delta_k(1) > \kappa^* + (N-k)h,&~~\mbox{if}~k=M+1,M+2,\cdots,N-1.
\end{eqnarray}
\end{lemma}
\begin{proof}
We only give the proof of \eqref{delta11-delta22}, and \eqref{delta11-delta22-000} can be proved similarly.

First of all, we have to prove
\begin{eqnarray}\label{delta11-delta22-step1}
\delta_k(\kappa^-) \ge \delta_k(1),~~k=1,2,\cdots,M.
\end{eqnarray}
In fact, from \eqref{def-delta-function}, we know that if we can prove
\begin{eqnarray}\label{delta11-delta22-step1-1}
\Theta_k(\kappa^-) \ge \Theta_k(1),~~k=1,2,\cdots,M,
\end{eqnarray}
then \eqref{delta11-delta22-step1} follows immediately.

Using \eqref{def-lambda-function} and
\begin{eqnarray}\label{cos-equ-k}
0 < \cos (k\pi h)< 1,~~k=1,2,\cdots,M,
\end{eqnarray}
we have
\begin{eqnarray}\label{def-inv-coskpih}
\frac{\omega}{\cos (k\pi h)} = \omega \lambda_k(\omega)  - \frac{h^2}{2 \cos (k\pi h)} - \eta(\omega),
\end{eqnarray}
where $\eta(\omega) = \sqrt{(\frac{2\omega + h^2}{2 \cos (k\pi h)})^2 - \omega^2}$.

Hence, from \eqref{def-delta-function} and \eqref{def-inv-coskpih}, we have
\begin{eqnarray}\label{Thet-k-omega-equ}
\Theta_k(\omega)
= \omega(\lambda_k(\omega) - \frac{Z_{M-1}(\lambda_k(\omega))}{Z_M(\lambda_k(\omega))}) - \frac{h^2}{2 \cos (k\pi h)} - \eta(\omega).
\end{eqnarray}

By \eqref{Z-j} and \eqref{def-lambda-function},
it follows that
\begin{eqnarray}\label{delta-11-trans-1}
w(\lambda_k(\omega) - \frac{Z_{M-1}(\lambda_k(\omega))}{Z_M(\lambda_k(\omega))})
= w \frac{\lambda_k(\omega) - (\lambda_k(\omega))^{-1}}{1 - (\lambda_k(\omega))^{-2M}}
= \frac{2 \eta(\omega)}{1 - (\lambda_k(\omega))^{-2M}}.
\end{eqnarray}
Thus, substituting \eqref{delta-11-trans-1} into \eqref{Thet-k-omega-equ}, we have
\begin{eqnarray}\label{delta-11-trans-1-A}
\Theta_k(\omega)= (\frac{2}{1 - (\lambda_k(\omega))^{-2M}} - 1) \eta(\omega) - \frac{h^2}{2 \cos (k\pi h)}.
\end{eqnarray}

From \eqref{def-lambda-function} and $\kappa^- \ge 1$, one can obtain
\begin{eqnarray}\label{lambda-k1kappa-}
&\lambda_k(1) \ge \lambda_k(\kappa^-) >1,~~k=1,2,\cdots,M.
\end{eqnarray}
Therefore, by using \eqref{delta-11-trans-1-A}, \eqref{lambda-k1kappa-}, the characteristic that $\eta(\omega)$ is increasing monotone with $w>0$ and $\kappa^-\ge 1$, we have \eqref{delta11-delta22-step1-1} established.

Then, we want to illustrate
\begin{eqnarray}\label{delta-22-trans-2-result}
\delta_k(1)  - \kappa^* > k h,~~k=1,2,\cdots,M.
\end{eqnarray}

Using \eqref{def-delta-function} and \eqref{cos-equ-k}, we get
\begin{eqnarray}\label{delta-22-trans-2}
\delta_k(1) - \kappa^* >  \Theta_k(1).
\end{eqnarray}

From \eqref{delta-11-trans-1-A} and $\lambda_k(1)>1$,
we have
\begin{eqnarray}\label{theta-k-1}
\Theta_k(1) > \eta(1) - \frac{h^2}{2 \cos (k\pi h)}.
\end{eqnarray}

For $k=1,2,\cdots,M$ and using \eqref{sin-inequiv-0}, it follows
\begin{eqnarray*}\nonumber
\frac{1}{\cos^2 (k\pi h)} >  1 + 4k^2h^2~~and~~\frac{1}{\cos (k\pi h)} \le  h^{-1},~~k=1,2,\cdots,M.
\end{eqnarray*}
From this, one can obtain
\begin{eqnarray}\label{delta-22-trans-2-R1-result}
\eta(1) \ge \sqrt{(1 + \frac{h^2}{2})^2 (1 + 4k^2h^2) - 1} > 2 k h~~ and ~~
\frac{h^2}{2 \cos (k\pi h)} &\le& \frac{h}{2}.
\end{eqnarray}

Combining \eqref{delta-22-trans-2}, \eqref{theta-k-1} and \eqref{delta-22-trans-2-R1-result}, we get
\eqref{delta-22-trans-2-result}.

In conclusion, using \eqref{delta11-delta22-step1} and \eqref{delta-22-trans-2-result}, we finish the proof of \eqref{delta11-delta22}.

\end{proof}

By using Lemma \ref{delta11-delta22-lemma}, it is obvious that
\begin{eqnarray}\label{deltak-delta1-kappa-000}
|\delta_k(\kappa^-)| \ge |\delta_k(1)| > \kappa^*.
\end{eqnarray}

\begin{lemma}\label{lemma-gammamm+1}
For any given $k=1,2,\cdots, N-1$, if $u \in W$, then
\begin{eqnarray}\label{def-gamma-m+1k-000}
&& |\bar{\gamma}_{M+1,k}|  \lesssim  \frac{h^{\frac{3}{2}}}{\sin \frac{k\pi h}{2}},\\\label{xq-cp:03-001-1}
&& |\bar{\gamma}_{M,k}\delta_k(1) + \bar{\gamma}_{M+1,k}\kappa^* |  \lesssim \frac{h^{\frac{3}{2}}}{\sin \frac{k\pi h}{2}}
( |\delta_k(1) - \kappa^*| + h |\delta_k(1)|).
\end{eqnarray}
Further, if the Assumption~{\bf I} is established, then
\begin{eqnarray}\label{def-gamma-m+1k-000-lemma-2}
&& |\bar{\gamma}_{M+1,k}|  \lesssim  \frac{h^{\frac{5}{2}}}{\sin \frac{k\pi h}{2}},\\\label{xq-cp:03-001-1-0002}
&& |\bar{\gamma}_{M,k}\delta_k(1) + \bar{\gamma}_{M+1,k}\kappa^* |  \lesssim
\frac{h^{\frac{5}{2}}}{\sin \frac{k\pi h}{2}} (|\delta_k(1) - \kappa^*| + h|\delta_k(1)|).
\end{eqnarray}
\end{lemma}

\begin{proof}

Using \eqref{expand-baruij-y-e}, \eqref{Rij-equation-B}, \eqref{def-CM1} and Lemma \ref{lemma-4}, similar to the proof of \eqref{inequiv-barRik} shows that
\begin{eqnarray}\label{inequiv-barRMplus1k-1}
|\bar{R}_{M+1,k}| \lesssim \frac{h^{\frac{3}{2}}}{\sin \frac{k\pi h}{2}}, ~~k=1,2,\cdots,N-1.
\end{eqnarray}

Similarly, if the Assumption~{\bf I} is established, by using \eqref{expand-baruij-y-e}, \eqref{corollary-1-eq1}, \eqref{corollary-1-eq2} and Lemma \ref{lemma-4}, we have
\begin{eqnarray}\label{inequiv-barRMplus1k-2}
|\bar{R}_{M+1,k}| \lesssim \frac{h^{\frac{5}{2}}}{\sin \frac{k\pi h}{2}}, ~~k=1,2,\cdots,N-1.
\end{eqnarray}

Therefore, from \eqref{def-gamma-m+1k}, \eqref{lemma-3-eq-4-e-2} and \eqref{inequiv-barRMplus1k-1}, we get \eqref{def-gamma-m+1k-000}. Similarly, using \eqref{inequiv-barRMplus1k-2}, we get \eqref{def-gamma-m+1k-000-lemma-2}.

The rest is to show that \eqref{xq-cp:03-001-1} and \eqref{xq-cp:03-001-1-0002} are established.

Using \eqref{def-gamma-m+1k} and \eqref{expand-baruij-y-e}, it follows
\begin{eqnarray}\label{lemma-bare2-Mk-eq-1}\nonumber
&& |\bar{\gamma}_{M,k}\delta_k(1) + \bar{\gamma}_{M+1,k}\kappa^*| \\\nonumber
&\le &  |\bar{R}_{M,k}\delta_k(1) + \bar{R}_{M+1,k}\kappa^*| +  \kappa^- |\bar{e}_{M-1,k}^{(1)}\cos (k\pi h)| |\delta_k(1)| + |\bar{e}_{M+2,k}^{(1)}\cos (k\pi h)| \kappa^*\\
&=& \beta_1 + \kappa^- |\bar{e}_{M-1,k}^{(1)}\cos (k\pi h)| |\delta_k(1)| + |\bar{e}_{M+2,k}^{(1)}\cos (k\pi h)| \kappa^*,
\end{eqnarray}
where
\begin{eqnarray}\label{def-beta-1}
\beta_1 = \sqrt{2h}|\sum\limits_{j=1}^{N-1} (R_{M,j} \delta_k(1) +  R_{M+1,j} \kappa^*) \sin (j\pi y_k)|.
\end{eqnarray}

By \eqref{lemma-3-eq-4-e-2} and Lemma \ref{delta11-delta22-lemma},
we have
\begin{eqnarray}\label{lemma-bare2-Mk-eq-3}
\kappa^- |\bar{e}_{M-1,k}^{(1)}\cos (k\pi h)| |\delta_k(1)| + |\bar{e}_{M+2,k}^{(1)}\cos (k\pi h)| \kappa^*
\lesssim |\delta_k(1)| \frac{h^{\frac{7}{2}}}{\sin \frac{k\pi h}{2}}.
\end{eqnarray}
Hence, substituting \eqref{lemma-bare2-Mk-eq-3} into \eqref{lemma-bare2-Mk-eq-1}, we can obtain
\begin{eqnarray}\label{lemma-bare2-Mk-eq-1-0002}
|\bar{\gamma}_{M,k}\delta_k(1) + \bar{\gamma}_{M+1,k}\kappa^*| \lesssim \beta_1 + |\delta_k(1)| \frac{h^{\frac{7}{2}}}{\sin \frac{k\pi h}{2}}.
\end{eqnarray}

Using \eqref{def-beta-1}, \eqref{Rij-equation-B}
and \eqref{def-CM1}, and noticing that \eqref{deltak-delta1-kappa-000}, we can obtain
\begin{eqnarray}\label{def-beta-1-2}\nonumber
\beta_1
&\lesssim& h^{\frac{3}{2}} |\delta_k(1) - \kappa^*| |\sum\limits_{j=1}^{N-1} C_{M}^{(1)}(y_j) \sin (j\pi y_k)| \\\nonumber
&&+ h^{\frac{5}{2}} |\sum\limits_{j=1}^{N-1} ( C_{M}^{(2)}(y_j) \delta_k(1)  +   C_{M+1}^{(2)}(y_j) \kappa^*)\sin (j\pi y_k)| + h^{\frac{5}{2}} |\delta_k(1)|\\\nonumber
&\le& h^{\frac{3}{2}} |\delta_k(1) - \kappa^*| |\sum\limits_{j=1}^{N-1} C_{M}^{(1)}(y_j) \sin (j\pi y_k)| \\
&&+ h^{\frac{5}{2}} |\delta_k(1)| (|\sum\limits_{j=1}^{N-1}  C_{M}^{(2)}(y_j) \sin (j\pi y_k)| +  |\sum\limits_{j=1}^{N-1}  C_{M+1}^{(2)}(y_j) \sin (j\pi y_k)| + h^{\frac{5}{2}} |\delta_k(1)|.~~
\end{eqnarray}
Further, if Assumption~{\bf I} holds, by similar derivation of \eqref{def-beta-1-2}, using \eqref{def-beta-1},
\eqref{corollary-1-eq1}, \eqref{corollary-1-eq2}, and noticing that \eqref{deltak-delta1-kappa-000} holds, we have
\begin{eqnarray}\label{def-beta-1-2-000}\nonumber
\beta_1
&\lesssim& h^{\frac{5}{2}} |\delta_k(1) - \kappa^*| |\sum\limits_{j=1}^{N-1} \tilde{C}_{M}^{(2)}(y_j) \sin (j\pi y_k)|\\
&& + h^{\frac{7}{2}} |\delta_k(1)| (|\sum\limits_{j=1}^{N-1} C_{M}^{(3)}(y_j) \sin (j\pi y_k)|  + |\sum\limits_{j=1}^{N-1}  C_{M+1}^{(3)}(y_j) \sin (j\pi y_k)|)  + h^{\frac{7}{2}} |\delta_k(1)|.~~~
\end{eqnarray}

Therefore, using \eqref{def-beta-1-2},
  \eqref{def-CM1}, \eqref{def-CM2}, \eqref{def-CM+12} and Lemma \ref{lemma-4}
we have
\begin{eqnarray}\label{lemma-bare2-Mk-eq-2}
&& \beta_1
\lesssim |\delta_k(1) - \kappa^*| \frac{h^{\frac{3}{2}}}{\sin \frac{k\pi h}{2}}
+ |\delta_k(1)|\frac{h^{\frac{5}{2}}}{\sin \frac{k\pi h}{2}}.
\end{eqnarray}

Similarly, if Assumption~{\bf I} is established,
from \eqref{def-beta-1-2-000}, \eqref{corollary-1-eq2}, \eqref{def-CM3}, \eqref{def-CM+13} and Lemma \ref{lemma-4}, we have
\begin{eqnarray}\label{lemma-bare2-Mk-eq-2-0002}
&& \beta_1
\lesssim |\delta_k(1) - \kappa^*| \frac{h^{\frac{5}{2}}}{\sin \frac{k\pi h}{2}} + |\delta_k(1)| \frac{h^{\frac{7}{2}}}{\sin \frac{k\pi h}{2}}.
\end{eqnarray}

Consequently, substituting \eqref{lemma-bare2-Mk-eq-2} and \eqref{lemma-bare2-Mk-eq-2-0002} into \eqref{lemma-bare2-Mk-eq-1-0002}, respectively, we have \eqref{xq-cp:03-001-1} and \eqref{xq-cp:03-001-1-0002} established.

\end{proof}

Using the above lemmas, we can get the estimation of $\|\vec{\bar{e}}^{(2)}_{k}\|_{\infty}(k=1,2,\cdots,N-1)$.
\begin{theorem}\label{theorem-2}
If $u \in W$, then
\begin{eqnarray}\label{eq-bar2-result-final}
\|\vec{\bar{e}}^{(2)}_{k}\|_{\infty}
\lesssim h^{\frac{1}{2}}  g(k),~~k=1,2,\cdots, N-1.
\end{eqnarray}
Further, if the Assumption~{\bf I} holds, then
\begin{eqnarray}\label{eq-bar2-result-final-0002}
\|\vec{\bar{e}}^{(2)}_{k}\|_{\infty}
\lesssim h^{\frac{3}{2}}  g(k),~~k=1,2,\cdots, N-1,
\end{eqnarray}
where
\begin{eqnarray}\label{eq-bar2-M-result-f}
g(k) &=& g(N - k) = g_1(k), ~k=1,2,\cdots,M, \\
\label{eq-bar2-M-result-f1}
 g_1(k) 
&=& \left\{\begin{array}{ll}
\frac{1}{k}, & \mbox{if}~k=1,2,\cdots,\lfloor \frac{N}{4} \rfloor,\\
\frac{1}{N-2k}, & \mbox{if}~ k=\lceil \frac{N}{4} \rceil,\lceil \frac{N}{4} \rceil+1,\cdots,M.
\end{array}\right.
\end{eqnarray}

\end{theorem}
\begin{proof}
We only present the proof of \eqref{eq-bar2-result-final}, and \eqref{eq-bar2-result-final-0002} can be proved similarly.

Note that $|\lambda_k(\omega)|>1(\omega = \kappa^-,1)$, by \eqref{express-bare2-result} and \eqref{Zj-property-000}, we have
\begin{eqnarray}\label{theorem-2-prof-key-1}
|\bar{e}^{(2)}_{i,k}| &\lesssim& \max\{|\bar{e}^{(2)}_{M,k}|, |\bar{e}^{(2)}_{M+1,k}|\},~~i=0,1,\cdots,M-1,M+2,M+3,\cdots,N.
\end{eqnarray}

Then, we will present the estimation of  $|\bar{e}_{M,k}^{(2)}|$ firstly.

Using \eqref{express-bare2-Mk} and \eqref{xq-cp:03-001-1}, we have
\begin{eqnarray}\label{express-bare2-Mk-inequiv-001}
|\bar{e}_{M,k}^{(2)}| \lesssim \frac{h^{\frac{3}{2}}}{\sin \frac{k\pi h}{2} |\cos (k\pi h)|}
(\frac{|\delta_k(1) - \kappa^*|}{|\delta_k(\kappa^-)\delta_k(1) - (\kappa^*)^2|} + \frac{h |\delta_k(1)|}{|\delta_k(\kappa^-)\delta_k(1) - (\kappa^*)^2|}).
\end{eqnarray}

By the first inequation of  \eqref{delta11-delta22} and \eqref{delta11-delta22-000}, we have
\begin{eqnarray*}
\frac{1}{|\delta_k(\kappa^-)\delta_k(1) - (\kappa^*)^2|}
\le  \frac{1}{|(\delta_k(1))^2 - (\kappa^*)^2|} = \frac{1}{|\delta_k(1) + \kappa^*| |\delta_k(1) - \kappa^*|}.
\end{eqnarray*}
From this, and using the second inequation of \eqref{delta11-delta22} and \eqref{delta11-delta22-000}, we can obtain
\begin{eqnarray}\label{key-eq-1-1}
\frac{|\delta_k(1) - \kappa^*|}{|\delta_k(\kappa^-)\delta_k(1) - (\kappa^*)^2|}
\le \frac{1}{|\delta_k(1) + \kappa^*|}
\lesssim \left\{\begin{array}{ll}
1, & \mbox{if}~k=1,2,\cdots,M,\\
\frac{1}{(N-k)h}, & \mbox{if}~k=M+1,M+2,\cdots,N-1,
\end{array}\right.
\end{eqnarray}
and
\begin{eqnarray}\label{key-eq-2-1}
\frac{|\delta_k(1)|}{|\delta_k(\kappa^-)\delta_k(1) - (\kappa^*)^2|} < \frac{1}{|\delta_k(1)| - \kappa^*} \lesssim
\left\{\begin{array}{ll}
\frac{1}{kh}, & \mbox{if}~k=1,2,\cdots,M,\\
\frac{1}{(N-k)h}, & \mbox{if}~k=M+1,M+2,\cdots,N-1.
\end{array}\right.
\end{eqnarray}

Substituting  \eqref{key-eq-1-1} and \eqref{key-eq-2-1} into \eqref{express-bare2-Mk-inequiv-001}, we have
\begin{eqnarray}\label{lemma-bare2-Mk-result-Wh}
&& |\bar{e}^{(2)}_{M,k}|
\lesssim h^{\frac{1}{2}} \tilde{g}(k),
\end{eqnarray}
where
\begin{eqnarray*}
\tilde{g}(k) &=&
\left\{\begin{array}{ll}
\frac{h}{\sin \frac{k\pi h}{2} |\cos (k\pi h)|}, & \mbox{if}~k=1,2,\cdots,M,\\
\frac{1}{(N-k) \sin \frac{k\pi h}{2} |\cos (k\pi h)|}, & \mbox{if}~k=M+1,M+2,\cdots,N-1.
\end{array}\right.~~~~~~
\end{eqnarray*}

Using the basic characteristics of trigonometric function, it is easy to verify
\begin{eqnarray}\label{def-tildeg}
\tilde{g}(k) \lesssim g(k),
\end{eqnarray}
where $g(k)$ is defined by \eqref{eq-bar2-M-result-f}.

Hence, substituting \eqref{def-tildeg} into \eqref{lemma-bare2-Mk-result-Wh}, we have
\begin{eqnarray}\label{lemma-bare2-Mk-result}
&& |\bar{e}^{(2)}_{M,k}|
\lesssim h^{\frac{1}{2}} g(k).
\end{eqnarray}

Secondly, by \eqref{express-bare2-Mk}, Lemma \ref{delta11-delta22-lemma}, \eqref{def-gamma-m+1k-000}, $\frac{h}{\sin \frac{k\pi h}{2}|\cos(k\pi h)|} \lesssim g(k)$ and \eqref{lemma-bare2-Mk-result}, we have
\begin{eqnarray}\label{eq-bare2-M+1}
 |\bar{e}^{(2)}_{M+1,k}|
 \le \frac{|\bar{\gamma}_{M+1,k}| + \kappa^* |\cos (k\pi h)|| \bar{e}^{(2)}_{M,k}|}{|\delta_k(1)||\cos (k\pi h) |}
 \lesssim \frac{1}{\sin \frac{k\pi h}{2} |\cos (k\pi h) |}h^{\frac{3}{2}}
 + |\bar{e}^{(2)}_{M,k}|
 \lesssim h^{\frac{1}{2}} g(k).
\end{eqnarray}

In conclusion, combining \eqref{theorem-2-prof-key-1}, \eqref{lemma-bare2-Mk-result} and \eqref{eq-bare2-M+1}, we finish the proof of \eqref{eq-bar2-result-final}.

\end{proof}

The following theorem gives the estimation of $\|\vec{e}^{(l)}\|_{\infty}(l=1,2)$.

\begin{theorem}\label{theorem-eij1}
If $u \in W$, then
\begin{eqnarray}\label{theorem-eij1-result}
\|\vec{e}^{(1)}\|_{\infty} \lesssim h^{2} |\ln h|, ~~\|\vec{e}^{(2)}\|_{\infty} \lesssim h |\ln h|.
\end{eqnarray}
Further, if the Assumption~{\bf I} holds, then
\begin{eqnarray}\label{theorem-eij2-result2}
\|\vec{e}^{(2)}\|_{\infty} \lesssim h^2 |\ln h|.
\end{eqnarray}
\end{theorem}
\begin{proof}
From \eqref{expand-uij-y-e}, we have
\begin{eqnarray}\label{theorem-5-3-equ-1}
|e_{i,j}^{(l)}|
\le \sqrt{2h} \sum\limits_{k=1}^{N-1} |\bar{e}_{i,k}^{(l)}|,~~l=1,2.
\end{eqnarray}

Using \eqref{eq-bar2-M-result-f}, we have
\begin{eqnarray}\label{theorem-5-3-equ-2}
\sum\limits_{k=1}^{N-1} g(k) = 2 (\sum\limits_{k=1}^{\lfloor\frac{N}{4}\rfloor}\frac{1}{k} + \sum\limits_{k=\lceil \frac{N}{4} \rceil}^{M}\frac{1}{N-2k})
\lesssim \sum\limits_{k=1}^{N-1}\frac{1}{k} \lesssim |\ln h|.
\end{eqnarray}

Therefore, substituting \eqref{lemma-3-eq-3-e-1-result}, \eqref{eq-bar2-result-final} and \eqref{eq-bar2-result-final-0002} into  \eqref{theorem-5-3-equ-1}, respectively, and using \eqref{theorem-5-3-equ-2}, we have \eqref{theorem-eij1-result} and \eqref{theorem-eij2-result2} established.

\end{proof}

Using Theorem \ref{theorem-eij1} and \eqref{decompose-eij}, we can obtain the main result of this paper.
\begin{theorem}\label{theorem-4}
If $u \in W$, then
\begin{eqnarray}\label{abs-eij-result}
\|\vec{e}\|_{\infty} \lesssim h |\ln h|.
\end{eqnarray}
Further, if the Assumption~{\bf I} holds, then
\begin{eqnarray}\label{abs-eij-result-0}
\|\vec{e}\|_{\infty} \lesssim h^2 |\ln h|.
\end{eqnarray}
\end{theorem}

\begin{remark}
Although the above theoretical estimations just prove that
the Assumption~{\bf I} is the sufficient condition
for scheme \eqref{2D-MACH-5pt-schemes} to reach the optimal asymptotic order,
the numerical experiments (see Table \ref{table-7-0} and Table \ref{table-1} in the following section) show that the assumptions also seems to be necessary.
\end{remark}

\section{\label{sec:5}Numerical experiments}\setcounter{equation}{0}

In this section, some typical numerical experiments are carried out
to experimentally study the accuracy of the five-point MACH-like
scheme \eqref{2D-MACH-5pt-schemes}.

\begin{example}\label{Benchmark-Pb2-2D-3}
Consider the simplified model of \eqref{Benchmark-Eq1} and \eqref{Benchmark-Eq3}, where the diffusion coefficient $\kappa^- = 10^4$, the exact solution $u(x,y) = \sin(\pi x) \sin(\pi y) ((x - \frac{1}{2})/\kappa + 1)$ does not satisfy \eqref{assumption-equ}.
\end{example}

The experiment results for solving Example
\ref{Benchmark-Pb2-2D-3} are given in Table \ref{table-7-0}, where
$\tilde{\kappa}_1 = \frac{\kappa^- + 1}{2}$ and $\tilde{\kappa}_2 =
\frac{2\kappa^-}{\kappa^- + 1}$.
It can be
observed 
that if the exact solution does not satisfy the assumption condition
\eqref{assumption-equ}, no matter whether 
the diffusion coefficients in the multiple material cells use
harmonic averaging or not, this scheme can not reach the asymptotic
optimal error estimate $O(h^2|\ln h|)$ in the maximum norm.

\begin{table}
\footnotesize \caption{Results for $u(x,y) = \sin(\pi x) \sin(\pi y) ((x - \frac{1}{2})/\kappa + 1)$ and $\kappa^- = 10^4$.}\label{table-7-0} \centering
\begin{tabular}{{|c|c|c|c|c|}}\hline
\multirow{2}{*}{$N$}        &  \multicolumn{2}{c|}{$\kappa^* = \tilde{\kappa}_1$} &\multicolumn{2}{c|}{$\kappa^* = \tilde{\kappa}_2$}\\\cline{2-5}
               &   $\|\vec{e}\|_{\infty}$ & ratio & $\|\vec{e}\|_{\infty}$ & ratio \\\hline
$33$ &     1.65E-02& --&      4.64E-02& --\\
$67$ &     7.81E-03& 2.12&    2.21E-02& 2.09\\
$135$&     3.79E-03& 2.06&    1.08E-02& 2.05\\
$271$&     1.87E-03& 2.03&    5.35E-03& 2.02\\
\hline
\end{tabular}
\end{table}


\begin{example}\label{Benchmark-Pb2-2D}
Consider the simplified model of \eqref{Benchmark-Eq1} and \eqref{Benchmark-Eq3}, where the diffusion coefficient $\kappa^- = 10^6$, the exact solution $u(x,y) = \frac{1}{\kappa}\sin(\pi x)\sin(\pi y) (x-\frac{1}{2}) (y-1) y (1+x^2 +y^2)$ satisfies \eqref{assumption-equ}.
\end{example}

The correponding experiment results
given in Table \ref{table-1} show that only when the exact solution
satisfies the assumption condition \eqref{assumption-equ} and
$\kappa^* = \tilde{\kappa}_2$,  the scheme can reach  the asymptotic
optimal error estimate $O(h^2|\ln h|)$, otherwise, the scheme can
only reach $O(h|\ln h|)$.

\begin{table}
\footnotesize \caption{Results for $u(x,y) = \frac{1}{\kappa}\sin(\pi x)\sin(\pi y) (x-\frac{1}{2}) (y-1) y (1+x^2 +y^2)$ and $\kappa^- = 10^6$.}\label{table-1} \centering
\begin{tabular}{{|c|c|c|c|c|}}\hline
\multirow{2}{*}{$N$}       &  \multicolumn{2}{c|}{$\kappa^* = \tilde{\kappa}_1$} &\multicolumn{2}{c|}{$\kappa^* = \tilde{\kappa}_2$}\\\cline{2-5}
               &   $\|\vec{e}\|_{\infty}$ & ratio & $\|\vec{e}\|_{\infty}$ & ratio \\\hline
$33$ &  5.80E-03& --&    4.70E-04& --\\
$67$ &  2.85E-03& 2.04&  1.14E-04& 4.12\\
$135$&  1.41E-03& 2.02&  2.81E-05& 4.06\\
$271$&  7.02E-04& 2.01&  6.98E-06& 4.03\\
\hline
\end{tabular}
\end{table}

%



All the results above-mentioned verify the correctness of the error theories.

In addition, a series of numerical experiments have carried out for the more general nine-point scheme \eqref{dis-9pt-scheme}.
The experimental results are similar to those of the five-point scheme, but it is omitted on account of space limitation.

\section{Conclusions}\label{sec:6}

In this paper, a kind of vertex-centered MACH-like FVM for stationary
diffusion problems with interface is constructed, and the estimates
of the local truncation error and global error have been
established, then the theoretical results are verified by numerical
experiments. It's worth pointing out that, if the exact solution
does not satisfy the assumption condition \eqref{assumption-equ},
the five-point MACH-like scheme with harmonic averaging  can
not reach the asymptotic optimal error estimate $O(h^2|\ln h|)$ in
the maximum norm. In the future, we hope to generalize this work to
more complex diffusion problems and schemes, such as 3D diffusion
problems and other schemes.

\section*{Acknowledgements}
The authors would like to thank Dr. Cunyun Nie from Hunan Institute
of Engineering for his helpful comments and suggestions. This work
is supported by the National Natural Science Foundation of China
(Grant Nos. 11571293, 11601462 and 61603322), Hunan Provincial
Natural Science Foundation of China (Grant No. 2016JJ2129), and Open
Foundation of Guangdong Provincial Engineering Technology Research
Center for Data Science (Grant No. 2016KF03).

%

\bibliographystyle{elsarticle-num}

\FloatBarrier

\end{document}